\documentclass[10pt]{amsart}

\usepackage{amsmath}
\usepackage{amsthm}
\usepackage{amssymb}
\usepackage{mathrsfs}
\usepackage{Ricci}

\theoremstyle{plain}
\newtheorem{thm}{Theorem}[section]
\newtheorem*{thmnn}{Theorem} 
\newtheorem*{mthm}{Main Theorem}
\newtheorem{lem}[thm]{Lemma}
\newtheorem{cor}[thm]{Corollary}
\newtheorem{prop}[thm]{Proposition}

\theoremstyle{definition}
\newtheorem{defn}[thm]{Definition}

\newcommand{\RR}{\mathbb{R}}

\newcommand{\HH}{\mathbb{H}}
\newcommand{\QH}{\mathcal{H}}
\newcommand{\PP}{\mathcal{P}}

\newcommand{\vphi}{\varphi}
\newcommand{\veps}{\varepsilon}
\newcommand{\Casimir}{\Upsilon}
\newcommand{\laplace}{\Delta}
\newcommand{\Hlaplace}{\mathcal{L}_0}
\newcommand{\Lie}{\mathscr{L}}
\newcommand{\inmult}{\lrcorner\,}

\newcommand{\parfrac}[1]{\frac{\partial}{\partial #1}}
\newcommand{\order}[1]{{\mathscr{O}_{#1}}}

\newcommand{\norm}[1]{\lVert #1 \rVert}

\renewcommand{\Im}{\textrm{Im}}

\DeclareMathOperator{\End}{End}

\DeclareMathOperator{\grad}{grad}

\begin{document}

\title[QC pseudohermitian normal coordinates]{Quaternionic contact pseudohermitian Normal coordinates}
\author{Christopher S. Kunkel}

\begin{abstract}
This paper constructs a family of coordinate systems about a point on a quaternionic contact manifold, called quaternionic contact pseudohermitian normal coordinates.  Once defined, conformal variations of the quaternionic contact structure induce changes on the coordinates which are studied in an effort to simplify the torsion and curvature at the center point.  These normalizations are constructed in the hope that they may help prove the quaternionic contact version of the Yamabe problem.
\end{abstract}

\maketitle

\section{Introduction} \label{sec:intro}

Quaternionic contact manifolds, first defined in \cite{Biquard:1999} and \cite{Biquard:2000},  appear naturally as the boundaries at infinity of asymptotically hyperbolic quaternionic manifolds.  In this way they generalize to the quaternion algebra the sequence of families of geometric structures that are the boundaries at infinity of real and complex asymptotically hyperbolic spaces.  In the real case, these manifolds are simply conformal manifolds, that is, manifolds with a family of Riemannian metrics that are all conformally related by smooth, positive functions.  In the complex case, the boundary structure is that of a CR manifold, a manifold with a contact structure and for each choice of contact form, a metric on the contact distribution, satisfying certain additional conditions related to those defining a complex manifold.

A quaternionic contact manifold is similar, having a type of contact structure, defined, not by a single non-vanishing $1$-form as in the CR case, but rather by an $\RR^3$-valued $1$-form.  On the kernel of this form, there is a metric which, when paired with the exterior derivatives of the three components of the contact form, defines a triple of almost complex structures satisfying the commutation relations of the unit quaternions $i$, $j$ and $k$.  Like the CR case, there is not a fixed contact form, but rather a conformal family of them, which of course gives rise to distinct but conformally related metrics on the contact distribution.  The almost complex structures defined by this varying family determine a three-dimensional subbundle of the endomorphism bundle of the contact distribution which may be compared to the three dimensional space of imaginary quaternions.

A natural question coming from this conformal freedom is the quaternionic contact Yamabe problem, named for its obvious similarity to the original Yamabe problem which asked if, given any metric on a compact manifold, there is a conformal metric of constant scalar curvature.  Many mathematicians have contributed to this problem since it was first posed in 1960; see \cite{LeeParker:1987} for a detailed overview and references.

Because CR manifolds have a similar conformal structure, and it is possible to define a well-adapted linear connection on such manifolds, there is a natural generalization of the Yamabe problem to the CR Yamabe problem.  Though more complicated than its conformal counterpart, the CR Yamabe problem has been solved using several of the techniques used in the proof of the conformal case.

In the same way, this paper provides a key step toward the solution of the quaternionic contact Yamabe problem, namely the construction of a coordinate system in which the invariants of the quaternionic contact structure are considerably simplified.

The primary tool for this is a generalized version of Riemannian normal coordinates, better adapted to the study of quaternionic contact manifolds.  The main ingredient in the construction of these coordinates is the fact that the tangent space of a quaternionic manifold has a natural \emph{parabolic} dilation, instead of the more common linear dilation seen in Riemannian geometry.  This suggests that the curves of interest used to define an exponential map from the tangent space at a point to the base manifold should incorporate this parabolic structure.  This gives rise to the following theorem.
\begin{thmnn}[pg.~\pageref{thm:coordinates}]
Let $(M,\nabla)$ be a manifold with connection whose tangent bundle decomposes as the direct sum of two  distributions, $H$ and $V$.  Choose any $q\in M$, and let $(X,Y)\in H_q \oplus V_q = T_q M$ be any tangent vector.  For any curve $\gamma$ on $M$ and any vector field $X$ along $\gamma$, let $D_t X$ denote the covariant derivative of $X$ along $\gamma$.   Define $\gamma_{(X,Y)}$ to be the curve satisfying
\begin{equation*}
D_t^2 {\dot\gamma}_{(X,Y)} = 0,\quad  \gamma_{(X,Y)}(0)=q, \quad {\dot\gamma}_{(X,Y)}(0)=X, \text{ and } D_t {\dot\gamma}_{(X,Y)}(0) = Y.
\end{equation*}
Then there are neighborhoods $0\in \mathcal{O}\subset T_qM$ and $q\in \mathcal{O}_M \subset M$ so that the function $\Psi : \mathcal{O} \to \mathcal{O}_M : (X,Y)\mapsto \gamma_{(X,Y)}(1)$  is a diffeomorphism, and satisfies the parabolic scaling $\Psi(tX,t^2Y)=\gamma_{(X,Y)}(t)$ wherever either side is defined.
\end{thmnn}

This theorem represents a generalization of the work in \cite{JerisonLee:1989}, from which this paper takes its inspiration.  There the authors construct a parabolic coordinate system on a CR manifold for the same purpose we have here.  Their result requires several other, very specific hypotheses and works only in the case of a CR manifold.  This new proof requires nothing more than a linear connection and decomposition of the tangent bundle by complementary distributions.

Using this theorem and a special frame at the center point $q$, I am able to construct a set of parabolic normal coordinates, so named for their similarity to the normal coordinates of Riemannian geometry with a parabolic, rather than linear scaling.  In particular we have the following theorem.
\begin{thmnn}[pg.~\pageref{thm:paraboliccoordinates}]
Let $M$ be a QC manifold with fixed pseudohermitian structure, and let $q\in M$ be any point.  Then there exist parabolic normal coordinates $(x^\alpha,t^i)$ about $q$ so that at $q$ the metric is the standard Euclidean metric, the contact forms are the standard forms on the quaternionic Heisenberg group and the connection $1$-forms vanish. Any two such coordinate systems centered at $q$ are related by a unique linear transformation in $Sp(n)Sp(1)$.
\end{thmnn}

Once these parabolic normal coordinates are defined, I explore the effect of a conformal change of contact structure on the curvature tensor of the Biquard connection, the standard linear connection of a quaternionic contact manifold.  Using the parabolic normal coordinates, I am able to define a function that, when used as the conformal factor, causes the symmetrized covariant derivatives of a certain tensor constructed from the curvature and torsion to vanish, as in the following theorem.
\begin{mthm}[pg.~\pageref{thm:symder_vanish}]  Let $M$ be a QC manifold.  For any $q\in M$ and any $N\geq 2$, there is a choice of pseudohermitian structure such that all the symmetrized covariant derivatives of $Q$ with total order less than or equal to $N$ vanish at $q$.  If we express the chosen pseudohermitian structure as a conformal multiple of another pseudohermitian structure, the $1$-jet of the conformal factor  at $q$ may be freely chosen.  Once this is fixed, the Taylor series of the conformal factor at $q$ is uniquely determined.
\end{mthm}

This idea originates in the work of Robin Graham, described in \cite{LeeParker:1987}.  Graham developed coordinates in the conformal case which showed that, at a point, the symmetrized covariant derivatives of the Ricci tensor can be made to vanish to high order by carefully choosing the conformal factor for a conformal manifold.  In their 1989 paper \cite{JerisonLee:1989} Jerison and Lee showed in a similar fashion that the symmetrized covariant derivatives of a tensor constructed from the CR curvature and torsion could be made to vanish at a point.

The structure of the tensor $Q$ is such that many of the curvature and torsion terms of the Biquard connection vanish at the center of normal coordinates when the symmetrized derivatives of $Q$ vanish to order four (see Theorem \ref{thm:mainthm}).  In fact, using invariance theory, I show that the only remaining term of weight less than or equal to $4$ that does not necessarily vanish at the center point is the squared norm of the quaternionic contact version of the Weyl tensor.


Section \ref{sec:background} reviews the relevant background of QC manifolds, including a number of important curvature and torsion identities.  Section \ref{sec:coords} defines QC pseudohermitian normal coordinates and the curvature-torsion tensor that can be made to vanish.  Finally, section \ref{sec:scal_invariants} shows how to use the normalization of section \ref{sec:coords} to simplify the curvature tensor at the origin of the QC normal coordinates.  Throughout the paper many of the results are proved in a similar fashion to the analogous statements in \cite{JerisonLee:1989}, and so the proofs of those results are omitted here.
\section{Background} \label{sec:background}

\subsection{Quaternionic contact manifolds}\label{sec:qc}

We begin by defining the focus of this paper, quaternionic contact manifolds.  These manifolds were first described in \cite{Biquard:2000}; however, in this paper we will use the definition provided in \cite{Vassilevetal:2006}, since many useful identities come from that paper.  In some sense, quaternionic contact manifolds are strictly pseudoconvex CR manifolds working with the quaternions instead of the complex numbers.  In fact the similarities are so strong that many of the proofs in section \ref{sec:coords} are almost identical to the corresponding proofs in \cite{JerisonLee:1989}.  Formally, we have the following definition. 
\begin{defn}
A \emph{quaternionic contact manifold}, or QC manifold,  is a $(4n+3)$-dimensional manifold $M$ with a $4n$-dimensional distribution, $H$, that satisfies the following properties:
\begin{itemize}
\item $H$ is the kernel of an $\RR^3$-valued 1-form $\eta = (\eta^1, \eta^2, \eta^3)$;
\item on $H$, there are almost complex structures $I^i$, $i=1,2,3$ that satisfy the commutation relations of the unit quaternions, i.e. \[(I^i)^2=I^1I^2I^3=-Id;\]
\item there is a sub-Riemannian metric $g$ on $H$ that satisfies \[2g(I^iX,Y)=d\eta^i(X,Y)\] for all vectors $X,Y$ in $H$ and each $i=1,2,3$.
\end{itemize}
\end{defn}

On a QC manifold, a choice of contact forms, metric and almost complex structures is called a \emph{QC pseudohermitian structure}. Such a choice is not unique.  First note that there is a conformal freedom, given by multiplying $g$ and the $\eta^i$ by the same smooth, positive function and leaving the almost complex structures unchanged.  On the other hand, within a given conformal class we have the following lemma from \cite{Vassilevetal:2006}.  Given the almost complex structures in the definition above, we let $Q=Span\{I^i\}_{i=1,2,3}$.
\begin{lem}
\begin{enumerate}
\item If $(\eta,I^i, g)$ and $(\eta,\hat{I}^i, \hat{g})$ are two QC structures on $M$, then $I^i = \hat{I}^i$ for $i=1,2,3$ and $g=\hat g$.
\item If $(\eta, Q, g)$ and $(\hat\eta, \hat Q, g)$ are two QC structures on $M$ with the same horizontal distribution, then $Q=\hat Q$ and $\hat\eta = \Psi \eta$ for some smooth $SO(3)$-valued function $\Psi$.
\end{enumerate}
\end{lem}

The presence of the three almost complex structures and their relation to the metric $g$ on $H$ provides an action of $Sp(n)Sp(1)$ on the bundle $H$.  An $Sp(n)Sp(1)$ frame for $H$ is an orthonormal frame $\xi_\alpha$, $\alpha=1,\ldots, 4n$ with $\xi_{4k+i+1}=I^i \xi_{4k+1}$ for $k=0,\ldots,n-1$ and $i=1,2,3$.

\subsubsection{The Biquard connection}\label{sec:biquardconn}

In his 2000 paper, Biquard defines a connection well suited to the study of QC manifolds, similar in spirit to the Tanaka-Webster connection on a CR manifold.  Let $Q\subset \End{H}$ be the span of the three almost complex structures. 
\begin{thm}[\cite{Biquard:2000}] \label{thm:biquardconn} Let $(M^{4n+3},\eta,g,Q)$ be a manifold with QC pseudohermitian structure, with $n>1$.  Then there exists a unique linear connection $\nabla$ with torsion $T$ and a unique distribution $V\subset TM$, complementary to $H$, such that:
\begin{itemize}
\item $\nabla$ preserves the decomposition $TM=H\oplus V$ and the metric $g$ on $H$;
\item given $X,Y\in H$, $T(X,Y)\in V$;
\item $\nabla$ preserves the $Sp(n)Sp(1)$ structure on $H$ (i.e. $Q$ is preserved);
\item given $R\in V$, $T(R,\cdot)|_H$ is an endomorphism of $H$ in $(\mathfrak{sp}(n)\oplus\mathfrak{sp}(1))^\perp\subset \mathfrak{gl}(4n)$;
\item there is natural isomorphism $\vphi:V\to Q$ and $\nabla \vphi=0$.
\end{itemize}
\end{thm}

The distribution $V$ may be described explicitly as the span of the three \emph{Reeb fields}, $R_1,\, R_2,\, R_3$, given by
\[ \eta^i(R_j) = \delta^i_j, \quad R_i\lrcorner d\eta^i|_H = 0,\]
where no summation is implied in the second equation.  In fact, more is true.  Letting $\tensor{\omega}{\down i \up j} = R_i\lrcorner d\eta^j|_H$, we have $\tensor{\omega}{\down i \up j} = -\tensor{\omega}{\down j \up i}$, and these are the restriction to $H$ of the connection 1-forms for $\nabla$ on $V$.  The isomorphism $\vphi$ is then simply $\vphi(R_i) = I_i$, and so these are also connection 1-forms on $Q$.

We can readily extend the metric $g$ to all of $TM$ by requiring that the Reeb fields be orthonormal, and $V \perp H$.  This metric depends only on the choice of $g$ and $\eta^i$, and in fact is given by $g\oplus \sum(\eta^i)^2$.  Further, from Theorem \ref{thm:biquardconn}, it is apparent that the extended metric is also parallel with respect to the connection.

Because we will be working with a carefully chosen frame in the later sections of the paper, it is now worthwhile to introduce some related notation.  For the remainder of this section we will work with a frame $\{ \xi_\alpha,\ R_i\}_{\alpha=1,\ldots,4n; i=1,2,3}$ where $\{\xi_\alpha\}$ is an $Sp(n)Sp(1)$ frame for $H$, and the $R_i$ are the three Reeb fields described above.  It is occasionally convenient to have a notation for the entire frame; therefore as necessary we may refer to $R_i$ as $\xi_{4n+i}$.  In order to have a consistent index notation we will use different letters for different ranges of indices.  In particular
\begin{gather*}
\alpha, \beta,\gamma,\ldots \in \{ 1, \ldots, 4n\},\\
i,j,k,\ldots \in \{1,2,3\}, \text{ and} \\
a,b,c, \ldots \in \{1,\ldots,4n+3\}.
\end{gather*}
For the dual basis we use the names $\theta^\alpha$ and $\eta^i$, and as above we may occasionally use  $\theta^{4n+i}=\eta^i$.  Further, throughout this paper we will obey the Einstein summation convention whenever possible.

Finally we note one last fact of importance, namely that both $H$ and $V$ are orientable.  The horizontal bundle is orientable since it admits an $Sp(n)Sp(1)\subset SO(4n)$ structure, and as noted above $Q$ has an $SO(3)$ structure, hence so does $V$.  Moreover, the natural volume form on $V$ is given by $\veps = \eta^1 \wedge \eta^2 \wedge \eta^3$, and this tensor provides a handy isomorphism between $V$ and $\bigwedge^2 V$ (or their duals).  Though we will not have much occasion to use them, we denote the volume form on $H$ by $\Omega$ and the volume form on $TM$ as $dv = \Omega \wedge \veps$.

Using the volume form $\veps^{ijk}$ and the metrics on $H$ and $V$, there is a convenient way to express composition of the almost complex structures. We will also require a few identities involving contractions of the volume form with itself, and so we list them here.
\begin{prop}\label{prop:ACSandVform}
The almost complex structures and volume form on $V$ satisfy the following identities:
\begin{gather}
\tensor{I}{\down i \up \alpha \down \gamma} \tensor{I}{\down j \up \gamma \down \beta} = -g_{ij}\delta^\alpha_\beta + \veps_{ijk}\tensor{I}{\up{k \alpha} \down \beta}; \label{eq:ACS}\\
\veps_{ijk}\veps^{ilm} = \delta_j^l \delta_k^m - \delta_k^l \delta_j^m,\quad \veps_{ijk}\veps^{ijl} = 2\delta_k^l. \label{eq:Vform}
\end{gather} 
\end{prop}

\subsubsection{Curvature and torsion identities} \label{sec:biquardcurv}

Because there are so many relations among the tensors defining a pseudohermitian structure, we can expect that there will be many relations between the curvature and torsion tensors of the Biquard connection.  In fact, we will see that there is a very close relation between the quaternionic contact torsion and the horizontal Ricci tensor.

We begin with the torsion.  The most basic identity it satisfies is for two vectors in $H$.  To wit, since $T(X,Y)\in V$ for $X,Y \in H$, 
\[ \tensor{T}{\up i \down{\alpha\beta}} = -\eta^i[\xi_\alpha,\xi_\beta] = d\eta^i(\xi_\alpha, \xi_\beta) = 2 g(I^i \xi_\alpha, \xi_\beta) = -2\tensor{I}{\up i \down{\alpha\beta}}, \text{ and } \tensor{T}{\up \alpha \down{\beta\gamma}} = 0.\]

Before we decompose the torsion tensor further, it is handy to introduce the Casimir operator $\Casimir:\End(H)\to\End(H)$ on a QC manifold, defined in \cite{Vassilevetal:2006}.  It is defined locally by
\[ \tensor{\Casimir}{\up {\alpha\gamma} \down {\beta\delta}} = \tensor{I}{\down i \up \alpha \down \beta} \tensor{I}{\up {i\gamma} \down\delta},\]
and satisfies the identity $\Casimir^2 = 2\Casimir + 3$.  It therefore has eigenvalues $3$ and $-1$.

The most interesting part of the torsion is the collection of endomorphisms \[T(R,\cdot)|_H:H\to H\] for a vector $R \in V$.  We summarize some results here.
\begin{prop}\label{prop:torsionprop}
The torsion tensor satisfies
\begin{itemize}
\item $\tensor{T}{\up i \down{\alpha\beta}} = -2\tensor{I}{\up i \down{\alpha\beta}}$;

\item $\tensor{T}{\up \alpha \down {i\alpha}} = 0= \tensor{T}{\up \alpha \down {i\beta}}\tensor{I}{\up {\beta} \down \alpha} $ for any $I\in Q$;

\item $\tensor{T}{\up \alpha\down {i\beta}} = \tensor{(T^0)}{\down i\up \alpha \down \beta} + \tensor{b}{\down i\up \alpha\down\beta}$ where $T^0$ is symmetric in the horizontal indices and $b$ is antisymmetric in the horizontal indices.

\item $\tensor{b}{\down i \up \alpha \down \beta} = \tensor{I}{\down i \up \alpha \down \gamma}\tensor{\mu}{\up \gamma \down \beta} = \tensor{\mu}{\up \alpha \down \gamma}\tensor{I}{\down i \up \gamma \down \beta}$ where $\mu$ is a symmetric, trace-free endomorphism of $H$ that commutes with any element of $Q$;

\item there is a symmetric, trace-free endomorphism of $H$ called $\tau$ such that \[\tensor{(T^0)}{\down i \up \alpha \down \beta} = \frac{1}{4}(\tensor{\tau}{\up \alpha \down \gamma} \tensor{I}{\down i \up \gamma \down \beta}  + \tensor{I}{\down i \up \alpha \down \gamma} \tensor{\tau}{\up \gamma \down \beta});\]

\item $\tensor{\Casimir}{\down {\alpha\beta}\up{\gamma\delta}} \mu_{\gamma_\delta} = 3\mu_{\alpha\beta}$ and $\tensor{\Casimir}{\down {\alpha\beta}\up{\gamma\delta}} \tau_{\gamma\delta} = -\tau_{\alpha\beta}$;

\item $\tensor{T}{\up \alpha \down {ij}} = d\theta^\alpha(R_i,R_j)$;

\item $\tensor{T}{\up k \down {ij}} = \lambda\, \tensor{\veps}{\down{ij}\up k}$ for a smooth function $\lambda$; and

\item $\tensor{T}{\up \alpha \down{\beta\gamma}} = \tensor{T}{\up i \down{j\alpha}}=0$.
\end{itemize}
\end{prop}
\begin{proof}
These properties are proved in \cite{Biquard:2000} and \cite{Vassilevetal:2006}.  
\end{proof}

\begin{cor}\label{cor:torsion}
The torsion tensor $\tensor{T}{\up \alpha\down {i\beta}}$ satisfies
\begin{equation}
\tensor{T}{\up \alpha\down {i\beta}} = \frac{1}{4}(\tensor{\tau}{\up \alpha \down \gamma} \tensor{I}{\down i \up \gamma \down \beta}  + \tensor{I}{\down i \up \alpha \down \gamma} \tensor{\tau}{\up \gamma \down \beta}) + \tensor{I}{\down i \up \alpha \down \gamma}\tensor{\mu}{\up \gamma \down \beta}.\label{eq:qctorsion}
\end{equation}
\end{cor}

Throughout this paper we use the index convention $\tensor{R}{\down{abc}\up d} = \theta^d( R(\xi_a,\xi_b)\xi_c)$ for the curvature tensor.  The (horizontal) Ricci tensor is $Ric = R_{\alpha\beta} = \tensor{R}{\down{\gamma\alpha\beta}\up\gamma}$ and the scalar curvature is $S=\tensor{R}{\down \alpha\up \alpha}$.  We list here some properties of the various curvature tensors.

\begin{prop}\label{prop:curvprop}
Let $\lambda$, $\tau$, and $\mu$ be as defined in Proposition \ref{prop:torsionprop}.  Then the curvature tensor satisfies
\begin{gather}
R_{\alpha\beta}= (2n+2)\tau_{\alpha\beta} + 2(2n+5)\mu_{\alpha\beta} + \frac{S}{4n}g_{\alpha \beta}; \label{eq:ricci} \\
S= -8n(n+2)\lambda;  \label{eq:scalar}\\
0= \tensor{\tau}{\down {\alpha\beta,}\up \beta}  - 6 \tensor{\mu}{\down {\alpha\beta,}\up \beta} - \frac{4n-1}{2}\veps^{ijk}\tensor{T}{\up \beta \down {jk}}I_{i\beta\alpha} - \frac{3}{16n(n+2)}S_{,\alpha}; \label{eq:taumuS} \\
0= \tensor{\tau}{\down {\alpha\beta,}\up \beta}  -  \frac{n+2}{2}\veps^{ijk}\tensor{T}{\up \beta \down {jk}}I_{i\beta\alpha} - \frac{3}{16(n+2)}S_{,\alpha}; \label{eq:tormu}\\
0 = \tensor{\tau}{\down {\alpha\beta,}\up \beta}  - 3\tensor{\mu}{\down {\alpha\beta,}\up \beta}  +2\veps^{ijk}\tensor{T}{\up \beta \down {jk}}I_{i\beta\alpha} - \tensor{R}{\down{\gamma i \beta}\up\gamma}\tensor{I}{\up{i\beta}\down \alpha}. \label{eq:VHricci}
\end{gather}
\end{prop}
\begin{proof}
These are all proved in \cite{Vassilevetal:2006}.  Equations \eqref{eq:ricci} and \eqref{eq:scalar} follow from Theorem 3.12, while equations \eqref{eq:taumuS}, \eqref{eq:tormu} and \eqref{eq:VHricci} follow from Theorem 4.8.
\end{proof}

\subsubsection{Conformal changes of QC pseudohermitian structure}\label{sec:confchange}

In the definition of a quaternionic contact manifold is the built-in possibility of a conformal change of structure.  By this we mean that there is a natural way to construct a new QC pseudohermitian structure from a given one, namely by multiplying the $\eta^i$ and the metric $g$ by a smooth positive function.  In particular, let $u \in C^\infty(M)$ and define $\tilde{\eta}^i = e^{2u}\eta^i$ and $\tilde g = e^{2u}g$.  Then $\bigcap_i \ker \tilde{\eta}^i =H$ and 
\[ d\tilde{\eta}^i = 2e^{2u}du\wedge \eta^i + e^{2u}d\eta^i,\]
from which we see that, restricted to $H$, \[d\tilde{\eta}^i(X,Y) = e^{2u}d\eta^i(X,Y) = 2e^{2u}g(I^iX,Y) = 2\tilde{g}(I^iX, Y).\]

From this conformal change, a whole cascade of changes takes place in the connection, torsion and curvature, which we record here for future use.
\begin{prop}\label{prop:confchange} Let $M$ be a QC manifold with pseudohermitian structure $\eta$.  Let $u\in C^\infty(M)$ and define $\tilde{\eta} = e^{2u}\eta$.  Then the Reeb fields for the new structure are given by 
\begin{equation*} \tilde{R}_i = e^{-2u}(R_i -  \tensor{I}{\down i \up \alpha \down \beta} u^\beta \xi_\alpha).\end{equation*} If we define $\tilde{\xi}_\alpha = \xi_\alpha$ and $\tilde{\theta}^\alpha = \theta^\alpha + \tensor{I}{\down i \up \alpha \down \beta}u^\beta \eta^i$, then $\tilde{\theta}^\alpha(\tilde{R}_i)=0$ and $\tilde{\eta}^i(\tilde\xi_\alpha)=0$.
Let $P_{-1}$ and $P_3$ denote the projections onto the $(-1)$- and $(3)$-eigenspaces of $\Casimir$.  Then the torsion and curvature tensors change as follows:
\begin{gather}
\tilde{\tau}_{\alpha\beta} =  \tau_{\alpha\beta} + P_{-1}(4u_\alpha u_\beta - 2u_{\alpha\beta}),\label{eq:tauchange}\\
\tilde{\mu}_{\alpha\beta} = \mu_{\alpha\beta} + P_3(-2u_\alpha u_\beta - u_{\alpha\beta}), \label{eq:muchange}\\
\tilde S \tilde{g}_{\alpha\beta} =  S g_{\alpha\beta}  - 16(n+1)(n+2) u_\gamma u^\gamma g_{\alpha\beta} - 8(n+2)\tensor{u}{\down \gamma \up \gamma}g_{\alpha\beta}.\label{eq:scalarchange}
\end{gather}
\end{prop}
\begin{proof}  These are proved in \cite{Vassilevetal:2006}, equations (5.11), (5.12) and (5.14) with the obvious change $h = \frac{1}{2}e^{-2u}$.
\end{proof}

\subsubsection{The quaternionic contact conformal curvature tensor} \label{sec:qcconfcurv}

In conformal geometry, the obstruction to conformal flatness is the well-studied Weyl tensor, the portion of the curvature tensor that is invariant under a conformal change of metric.  It is this tensor, or rather its vanishing, that determines if a conformal manifold is locally conformally equivalent to the standard sphere.

Likewise, in the CR case, the tensor which determines local CR equivalence to the CR sphere is the Chern tensor, also determined by the curvature of the Tanaka-Webster connection.  And just as in the conformal case, it is the key to finding the  appropriate bound for the CR Yamabe invariant on a CR manifold.

Something similar appears in the QC case, dubbed the quaternionic contact conformal curvature by Stefan Ivanov and Dimiter Vassilev in \cite{Vassilevetal:2007}.  In their paper they define a tensor $W^{qc}$ (which we will hereafter refer to as simply $W$) and prove that it is the conformally invariant portion of the Biquard curvature tensor.  Moreover, if it vanishes, they prove that the QC manifold is locally QC equivalent to the quaternionic Heisenberg group.  Since the quaternionic Heisenberg group and the QC sphere are locally equivalent, this tensor clearly plays the role of the Weyl or Chern tensors.  

In the conformal setting, for a given metric, the full Riemann curvature tensor can be expressed in terms of the Weyl tensor and the Ricci curvature.  Analogously, in the  QC case, the horizontal Riemann curvature tensor is described in terms of $W$ and a tensor $L$ which is itself determined by the horizontal Ricci tensor.  In particular,
\begin{equation} \label{eq:defineL} L_{\alpha\beta} =  \frac{1}{2} \tau_{\alpha\beta} + \mu_{\alpha\beta} + \frac{S}{32n(n+2)} g_{\alpha\beta},\end{equation}
where $\tau$ and $\mu$ are the tensors described in propositions \ref{prop:torsionprop} and \ref{prop:curvprop}.  Since $\tau$, $\mu$ and $S$ may be recovered from $L$ using the Casimir operator and the metric, their vanishing is equivalent to the vanishing of $L$.

Equation (4.8) of \cite{Vassilevetal:2007} expresses $W$ in terms of the curvature tensor and the tensor $L$ as
\begin{multline} \label{eq:QCconfcurv}
W_{\alpha\beta\gamma\delta} = R_{\alpha\beta\gamma\delta} + g_{\alpha\gamma}L_{\beta\delta} - g_{\alpha\delta}L_{\beta\gamma} + g_{\beta\delta}L_{\alpha\gamma} - g_{\beta\gamma}L_{\alpha\delta} \\
+ I_{i\alpha\gamma}L_{\beta\rho}\tensor{I}{\up{i\rho}\down\delta} - I_{i\alpha\delta}L_{\beta\rho}\tensor{I}{\up{i\rho}\down\gamma} + I_{i\beta\delta}L_{\alpha\rho}\tensor{I}{\up{i\rho}\down\gamma} - I_{i\beta\gamma}L_{\alpha\rho}\tensor{I}{\up{i\rho}\down\delta}\\
+ \frac{1}{2}( I_{i\alpha\beta}L_{\gamma\rho}\tensor{I}{\up{i\rho}\down \delta} - I_{i\alpha\beta}L_{\rho\delta}\tensor{I}{\up{i\rho}\down\gamma} + I_{i\alpha\beta}L_{\rho\sigma}\tensor{I}{\down j \up \rho \down \gamma} \tensor{I}{\down k \up \sigma \down \delta} \veps^{ijk}) \\
+ I_{i\gamma\delta}L_{\alpha\rho}\tensor{I}{\up{i\rho}\down \beta} - I_{i\gamma\delta}L_{\rho\beta}\tensor{I}{\up{i\rho}\down\alpha} + \frac{1}{2n} \tensor{L}{\down \rho \up \rho} I_{i\alpha\beta}\tensor{I}{\up i \down{\gamma\delta}}.
\end{multline}
Therefore, the QC conformal curvature equals the horizontal curvature tensor precisely when the tensor $L$ vanishes.

\subsection{The quaternionic Heisenberg group} \label{sec:quatheis}

We close this section with a brief review of the quaternionic Heisenberg group, a non-compact $4n+3$ manifold with a QC structure.  Let $\HH$ denote the quaternion algebra and $\HH^n$ the right $\HH$-module of $n$-tuples of quaternions.  Then the quaternionic Heisenberg group, $\QH^n$, is diffeomorphic to $\HH^n\oplus \Im \HH$ with group law
\[ (p_1,\omega_1)\cdot (p_2,\omega_2) = \big(p_1+p_2, \omega_1 + \omega_2 + 2\Im (p_1,p_2)\big),\]
where $(p_1,p_2) = \sum_{i=1}^n p_1^i (p_2^i)^\ast$ is the standard hermitian inner product on $\HH^n$.

Writing $p=(p^\alpha) = (w^\alpha+x^\alpha i + y^\alpha j + z^\alpha k)$ and $\omega = ri + sj+tk$, the left invariant $1$-forms on $\QH^n$ are
\begin{gather*}
dw^\alpha,\ dx^\alpha,\ dy^\alpha,\ dz^\alpha,\\
\eta^1 = \frac{1}{2}dr - \sum_\alpha x^\alpha dw^\alpha - w^\alpha dx^\alpha + z^\alpha dy^\alpha - y^\alpha dz^\alpha,\\
\eta^2 =  \frac{1}{2} ds - \sum_\alpha y^\alpha dw^\alpha - z^\alpha dx^\alpha - w^\alpha dy^\alpha + x^\alpha dz^\alpha,\\
\eta^3 = \frac{1}{2}dt - \sum_\alpha z^\alpha dw^\alpha + y^\alpha dx^\alpha - x^\alpha dy^\alpha - w^\alpha dz^\alpha.
\end{gather*}
Dual to these are the left invariant vector fields
\begin{gather*}
W_\alpha = \parfrac{w^\alpha} + 2\sum_\alpha x^\alpha \parfrac{r} + y^\alpha \parfrac{s} + z^\alpha \parfrac{t},\\
X_\alpha = \parfrac{x^\alpha} + 2\sum_\alpha -w^\alpha \parfrac{r} - z^\alpha \parfrac{s} + y^\alpha \parfrac{t},\\
Y_\alpha = \parfrac{y^\alpha} + 2\sum_\alpha z^\alpha \parfrac{r} - w^\alpha \parfrac{s} - x^\alpha \parfrac{t},\\
Z_\alpha = \parfrac{z^\alpha} + 2\sum_\alpha -y^\alpha \parfrac{r} + x^\alpha \parfrac{s} - w^\alpha \parfrac{t},\\
 R = 2 \parfrac{r}, \ S= 2 \parfrac{s}, \  T=2 \parfrac{t}. 
\end{gather*}
The QC structure on $\QH^n$ is given by declaring the left invariant vector fields above to be orthonormal and using the given $\eta^i$ as the contact forms.
The almost complex structures are then given by
\begin{gather*}
I^1: W_\alpha\mapsto X_\alpha,\quad X_\alpha \mapsto -W_\alpha, \quad Y_\alpha \mapsto Z_\alpha, \quad Z_\alpha \mapsto -Y_\alpha;\\
I^2: W_\alpha \mapsto Y_\alpha, \quad X_\alpha \mapsto -Z_\alpha, \quad Y_\alpha \mapsto -W_\alpha, \quad Z_\alpha \mapsto X_\alpha;\\
I^3=I^1\circ I^2.
\end{gather*}

The horizontal and vertical subbundles of $T\QH^n$ are given by \[\mathfrak{h} = Span\{W_\alpha,\ X_\alpha,\ Y_\alpha,\ Z_\alpha\} \text{ and } \mathfrak{v}=Span\{R,\ S,\ T\}\] and of course $T\QH^n = \mathfrak{h} \oplus \mathfrak{v}$.  The Biquard connection is given by declaring these left invariant vector fields to be parallel, and so the flat model of QC geometry is exactly the quaternionic Heisenberg group.  Finally we note the important fact that parabolic dilations $\delta_a(X,R) = (aX, a^2R)$ for $a\in \RR$ are automorphisms for the Lie group $\QH$ and hence also its Lie algebra.
\section{Coordinate constructions} \label{sec:coords}

In this section we construct a version of normal coordinates that are adapted to QC geometry in much the same way that standard normal coordinates are adapted to the study of Riemannian geometry.

\subsection{Parabolic normal coordinates}

We begin with a general theorem that constructs ``parabolic geodesics''; that is, curves that satisfy an invariant differential equation with a parabolic-type scaling.  By way of motivation, recall that $\RR^n$, and also its tangent spaces at every point, come equipped with a natural dilation that sends a vector $v$ to $sv$, for any real scalar $s$.  If we consider these vectors to be based at the origin $0\in\RR^n$, then by moving from $0$ in the direction of the vector $v$ for time $s$ we arrive at the standard parametrization of a line, $s\mapsto sv$.  Further, it is a simple matter to see that this line is uniquely determined by the initial value problem,
\[ \ddot{\gamma}=0, \quad \gamma(0)=0, \quad \dot{\gamma}(0)=v.\]
Using this as a guide, a geodesic on a manifold with linear connection is a curve satisfying the following initial value problem for a fixed $X\in T_qM$,
\[ D_t \dot\gamma_X = 0, \quad \gamma_X(0)=q \in M, \quad \dot\gamma_X(0)=X \in T_qM,\]
where for any curve $\gamma$ on $M$ and any vector field $X$ along $\gamma$, we let $D_t X$ denote the covariant derivative of $X$ along $\gamma$.    Notice that the standard dilations on $\RR^n$ interact naturally with a parametrized geodesic by
\[ \gamma_{sX}(t) = \gamma_X(st).\]  Further, these dilations are Lie algebra homomophisms for the commutative Lie algebra structure on $\RR^n$.

Just as the model for Riemannian geometry is $\RR^n$, the model for quaternionic contact geometry is the quaternionic Heisenberg group, described in section \ref{sec:quatheis}.  The quaternionic Heisenberg group has a  family of parabolic dilations $(x,t)\mapsto (sx,s^2t).$  If we start from a point  $0\in\QH^n$ and travel along the curve $s\mapsto (sv,s^2a)$, we trace out a parabola.  This serves as a guide for our notion of parabolic geodesics on a QC manifold.

As with a line, there is a simple expression for  a parabola in terms of a differential equation, namely
\[ \dddot{\gamma}_{(v,a)} = 0, \quad \gamma_{(v,a)}(0)=0, \quad \dot{\gamma}_{(v,a)}(0)=v, \quad \ddot{\gamma}_{(v,a)}(0) = a.\]
Extending this notion to a manifold with a linear connection produces curves that can rightly be called \emph{parabolic geodesics}, i.e. that satisfy a natural parabolic scaling \[\gamma_{(sv,s^2a)}(t)=\gamma_{(v,a)}(st).\]   By appropriately restricting our initial conditions, we can show that there is a parabolic version of the geodesic exponential map called the \emph{parabolic exponential map}.

The following theorem carries out this procedure.  It represents a generalization and improvement of the argument given by Jerison and Lee in \cite[Theorem 2.1]{JerisonLee:1989}, which is specific to strictly pseudoconvex CR manifolds.  In particular, Theorem \ref{thm:coordinates} requires no assumptions on the manifold except a direct sum decomposition of the tangent bundle by two complementary distributions, which is satisfied by both CR and QC manifolds.  In fact this theorem also generalizes the proof of the existence of geodesics on a manifold with connection, simply by assuming that the bundle $V$ is the zero section of $TM$.

\begin{thm} \label{thm:coordinates}
Let $(M,\nabla)$ be a manifold with connection whose tangent bundle decomposes as the direct sum of two  distributions, $H$ and $V$.  Choose any $q\in M$, and let $(X,Y)\in H_q \oplus V_q = T_q M$ be any tangent vector.  Define $\gamma_{(X,Y)}$ to be the curve beginning at $q$ satisfying
\begin{equation}\label{eq:curveeq}
D_t^2 {\dot\gamma}_{(X,Y)} = 0,\quad  \gamma_{(X,Y)}(0)=q, \quad {\dot\gamma}_{(X,Y)}(0)=X, \text{ and } D_t {\dot\gamma}_{(X,Y)}(0) = Y.
\end{equation}
Then there are neighborhoods $0\in \mathcal{O}\subset T_qM$ and $q\in \mathcal{O}_M \subset M$ so that the function $\Psi : \mathcal{O} \to \mathcal{O}_M : (X,Y)\mapsto \gamma_{(X,Y)}(1)$  is a diffeomorphism, and satisfies the parabolic scaling $\Psi(tX,t^2Y)=\gamma_{(X,Y)}(t)$ wherever either side is defined.
\end{thm}

\begin{proof}
First, note that for any two tangent vectors $X$ and $Y$, there is a unique smooth curve satisfying \eqref{eq:curveeq}.  This follows from the standard existence and uniqueness results on systems of ODEs applied to the coordinate form of the equation.  In particular, in any local coordinates $\{x^i\}$, centered at $q$, we let $\Gamma_{ij}^k = dx^k(\nabla_{\partial_i}\partial_j)$ be the Christoffel symbols of $\nabla$.  Then
\begin{equation}\label{eq:coordparabola}
(D_t)^2 \dot\gamma = \left( \dddot{\gamma}^k + \ddot{\gamma}^i\dot{\gamma}^j \Gamma_{ij}^k + 2 \dot{\gamma}^i \ddot{\gamma}^j \Gamma_{ij}^k + \dot{\gamma}^i\dot{\gamma}^j\dot{\gamma}^l \partial_l \Gamma_{ij}^k + \dot{\gamma}^i\dot{\gamma}^l\dot{\gamma}^m\Gamma_{lm}^j\Gamma_{ij}^k\right)\partial_k,
\end{equation}
so \eqref{eq:curveeq} is a third order, nonlinear system of ODEs with smooth coefficients.

Now for any $a\in \RR$, define $\sigma(t) = \gamma_{(X,Y)}(at)$.  Then $\dot{\sigma}(t) = a \dot{\gamma}(at)$, $D_t \dot\sigma = a^2 D_t\dot\gamma$ and $(D_t)^2\dot\sigma = a^3 (D_t)^2\dot{\gamma}=0$.  Thus, by uniqueness of solutions, $\sigma(t) = \gamma_{(X,Y)}(at) = \gamma_{(aX,a^2Y)}(t)$, which shows the parabolic scaling.

Now, let $\pi:E = TM \oplus TM \to M$ be the Whitney sum of $TM$ with itself.  Define a vector field $P$ on $E$ by
\begin{equation}\label{eq:parabolicvf}
P_{(p,X,Y)}f= \left.\frac{d}{dt}\right|_{t=0}  f( \gamma_{(X,Y)}(t), \dot{\gamma}_{(X,Y)}(t), D_t \dot\gamma_{(X,Y)}(t) ),
\end{equation}
for any function $f\in C^\infty(E,\RR)$.
Let $\{x^i\}$ be coordinates on  $M$, and take fiber coordinates $\{\eta^i\}$ and $\{\xi^i\}$ on $E$ where $\eta^i(p,X,Y) = dx^i(X)$ and $\xi^i(p,X,Y) = dx^i(Y)$.  Then in these coordinates, if we let $(\gamma_{(X,Y)}, \dot{\gamma}_{(X,Y)}, D_t\dot{\gamma}_{(X,Y)})  =(x^i, \eta^i, \xi^i)$, we can write $P_{(p,X,Y)}f$ as
\begin{align}
P_{(p,X,Y)}&= \dot{x}^k \frac{\partial}{\partial x^k}f + \dot{\eta}^k \frac{\partial}{\partial \eta^k}f + \dot{\xi}^k \frac{\partial}{\partial \xi^k}f \notag\\
	&=\eta^k \frac{\partial}{\partial x^k}f + (\xi^k - \eta^i\eta^j\Gamma_{ij}^k) \frac{\partial}{\partial \eta^k}f - \eta^i\xi^j\Gamma_{ij}^k \frac{\partial}{\partial \xi^k}f \label{eq:parabolicvfcoords}
\end{align}
In this formula, the $\Gamma_{ij}^k$ are the Christoffel symbols of $\nabla$, lifted to be constant on the fibers of $E$, and we have used the fact that $(D_t)^2\dot\gamma\equiv 0$.  This expression shows that $P$ is smooth, that integral curves of $P$ project onto solutions of \eqref{eq:curveeq}, and that solutions of \eqref{eq:curveeq} lift to integral curves of $P$ by $\rho(t) = (\gamma(t), \dot\gamma(t), D_t\dot\gamma(t))$.

The flow $\theta$ of $P$ is defined on some open subset of $\mathcal{O} \subset \RR\times E$ containing $\{0\}\times E$.  Thus from the above, $\gamma_{(X,Y)}(t)=\pi \circ \theta(t,(q,X,Y))$.

Let $\iota:T_qM \to E$ be the inclusion sending $X \in H_q$ to $(X,0)\in E_q$ and $Y\in V_q$ to $(0,Y)\in E_q$.  Then $\Psi(X,Y) = \pi \circ \theta(1, (q, \iota(X+Y)))$, which shows that $\Psi$ is smooth on some open set in $\iota^{-1}(\mathcal{O})$ to $M$.

We will now show that $\Psi$ is a diffeomorphism by showing that $\Psi_*$ is the identity map on $H$ and one half the identity map on $V$.  Given $X\in H_q$, we have
\[ \Psi_* X = \left.\frac{d}{dt}\right|_{t=0}  \Psi(tX,0) = \left.\frac{d}{dt}\right|_{t=0}  \gamma_{(tX,0)}(1) = \left.\frac{d}{dt}\right|_{t=0}  \gamma_{(X,0)}(t) = X.\]
For $Y\in V_q$, let $\alpha(t)$ be the $\nabla$-geodesic with initial velocity $Y$, and define $\beta(s) = \alpha(s^2t)$.  Then $\dot{\beta}(s) = 2st\dot{\alpha}(s^2t)$, $D_s\dot\beta(s) = 2t\dot\alpha(s^2t)$ and $(D_s)^2\dot\beta(s) = 0$.  Thus $\beta(s) = \gamma_{(0,2tY)}(s)$ and 
\begin{multline*}
 \Psi_* Y = \left.\frac{d}{dt}\right|_{t=0}  \Psi(0,tY) = \left.\frac{d}{dt}\right|_{t=0}  \gamma_{(0,tY)}(1) \\= \left.\frac{d}{dt}\right|_{t=0} \gamma_{(0,2tY)}\Big(\frac{\sqrt{2}}{2}\Big)= \left.\frac{d}{dt}\right|_{t=0}  \alpha(t/2) =\frac{1}{2}Y.
\end{multline*}

Since $\Psi_*$ is invertible, $\Psi$ is a diffeomorphism from some subset of $\iota^{-1}(\mathcal{O})\subset T_q M$ to $M$.  Relabeling our open sets if necessary, the theorem is proved.
\end{proof}

As a matter of interest, we note here that this proof extends readily to higher order ``polynomial geodesics'' that satisfy a higher order scaling condition, defined by the equations
\[
 D_t^{n}\dot\gamma = 0,\quad 
\gamma(0)=p,\quad  D_t^{k}\dot\gamma(0)= X_k, \ k=0,\ldots, n-1.
\]
Further, if we have a direct sum decomposition $TM = \oplus_{k=0}^{n-1} V^k$, and $X_k \in V^k|_p$, we have a diffeomorphism from a neighborhood of $0 \in TM$ to a neighborhood of $p\in M$ defined in the obvious way.

We will use the diffeomorphism described to provide a useful coordinate system on a QC manifold with a given pseudohermitian structure.  Before we do that, we will need the following lemma which allows us to take a frame for $T_qM$ and construct a local frame on a neighborhood of $q$, compatible with the parabolic exponential map.

\begin{lem}\label{lem:smoothvf} If $Z$ is any vector field parallel along all parabolic geodesics beginning at $q$, then $Z$ is smooth.
\end{lem}
\begin{proof}
Let $\{x^a\}$ be local coordinates centered at $q$, and let \[Z^a(s,X,Y)=dx^a(Z|_{\gamma_{(X,Y)}(s)}).\]  Then along every parabolic geodesic $\gamma$, $Z$ satisfies
\[ \partial_s Z^c(s,X,Y) + \dot{\gamma}_{(X,Y)}^a(s) Z^b(s,X,Y) \Gamma_{ab}^c(\gamma_{(X,Y)}(s)) = 0.\]
Since $\gamma_{(X,Y)}$ depends smoothly on $s$, and $X$, $Y$ and $\Gamma_{ij}^k$ depend smoothly on the coordinates, we see that $Z^a$ is a smooth function of its parameters.  Thus $Z_p = Z^a(1,\Psi^{-1}(p))\partial_a$ is a smooth vector field on a neighborhood of $q$.
\end{proof}

Now, let us return to the QC case.  Let $\{R_i\}$ be an oriented orthonormal frame for $V_q$, and let $\{I_i\}$ be the associated almost complex structures.  Choose an orthonormal basis $\{\xi_\alpha\}$ for $H_q$ so that $\xi_{4k+i+1} = I_i \xi_{4k+1}$ for $k=0,\ldots,n-1$.  Extending these vectors to be parallel along parabolic geodesics beginning at $q$, we have a smooth local frame for $TM = H\oplus V$.  Define the dual $1$-forms $\{\theta^\alpha,\eta^i\}$ by $\theta^\alpha(\xi_\beta) = \delta_\beta^\alpha$, $\theta^\alpha(R_i)=0$, $\eta^i(\xi_\alpha)=0$ and $\eta^i(R_j) = \delta_j^i$.  Finally, we extend the almost complex structures by defining $I_i \xi_{4k+1} = \xi_{4k+i+1}$ for each $k$.  Then each of these $1$-forms and almost complex structures is also parallel along parabolic geodesics, and all are parallel at $q$.  Using this frame and coframe, we have the following lemma. 
\begin{lem} For the frame, coframe and almost complex structures defined above, and for all vectors $X,Y\in H$, we have
\[ d\eta^i(X,Y) = 2g(I^i X, Y).\]
Thus the $\eta^i$,$I^i$ and $g$ form a QC pseudohermitian structure on $M$.
\end{lem}
\begin{proof}
Let $g$,  $\tilde{\eta}^i$ and $\tilde{I}^i$ be the metric,  contact $1$-forms and almost complex structures defining a QC pseudohermitian structure near $q$.  By a constant coefficient rotation, we may assume that $\tilde{\eta}^i|_q = \eta^i|_q$.  Since the connection preserves the metric, the orthogonality relations of the $\eta^i$ are preserved by parallel translation.  Thus, both $\{\eta^i\}$ and $\{\tilde{\eta}^i\}$ are oriented orthonormal $V$-coframes, and so are related by an orthogonal transformation at each point.  Since both frames are smooth, the transformation is smooth, and since the determinant is a continuous function on a connected set with values in $\pm 1$, which equals $1$ at $q$, the transformation actually lies in $SO(3)$.  Thus the $\eta^i$ and $I^i$ form a QC pseudohermitian structure with $g$.
\end{proof}
The frame and coframe constructed above will be called a \emph{special frame} and a \emph{special coframe}.

Given any special frame, we may define a coordinate map on a neighborhood of $q$ by composing the inverse of $\Psi$ with the map $\lambda: T_qM \to \RR^{4n+3}:X \mapsto (x^\alpha,t^i)=(\theta^\alpha(X),\eta^i(X))$.  These coordinates will be called \emph{QC pseudohermitian normal coordinates}, or pseudohermitian normal coordinates when no confusion can arise.  

With these definitions in mind, our index convention defined in section \ref{sec:background} is hereby refined to refer to a special frame and coframe henceforth.

\subsubsection{Parabolic Taylor expansions} \label{sec:parabolictaylor}

Returning to our analogy between scaling operators on $\RR^n$ and the quaternionic Heisenberg group, we recall that the generator for the standard dilation $x\mapsto sx$ on $\RR^n$ is the Euler vector field $X = x^i \partial_i$, and a tensor field $\vphi$ is called homogeneous of order $m$ if $\Lie_X \vphi = m \vphi$.  In the setting of parabolic dilations on the quaternionic Heisenberg group, the generator of $\delta_s: (x,t)\mapsto(sx, s^2t)$ is the vector field $P = x^\alpha \partial_\alpha + 2 t^i \partial_i$.  Note that in the notation of section \ref{sec:quatheis}, we can express $P$ in terms of the left invariant vector fields on $\QH$ as 
\[ P = w^\alpha W_\alpha + x^\alpha X_\alpha + y^\alpha Y_\alpha z^\alpha Z_\alpha + rR + sR + tT.\]  As in the Euclidean setting we say a tensor field $\vphi$ is homogeneous of order $m$ if $\Lie_P \vphi = m\vphi$.  For an arbitrary tensor field, we denote by $\vphi_{(m)}$ the part of the tensor that is homogeneous of order $m$.

Given a QC manifold and pseudohermitian coordinates as above centered at a point $q \in M$, we may define the vector $P$ in these coordinates.  Then next lemma shows how $P$ is related to the special frame that we have constructed.

\begin{lem}
Let $P$ be the vector described in pseudohermitian normal coordinates by $P = x^\alpha \partial_\alpha + 2t^i \partial_i$, and let $\tensor{\omega}{\down a \up b}$ be the connection $1$-forms for the Biquard connection. Then 
\[ \theta^\alpha(P) = x^\alpha,\  \eta^i(P) = t^i,\text{ and } \tensor{\omega}{\down a \up b}(P)=0.\]
Thus $P = x^\alpha \xi_\alpha + t^i R_i$.
\end{lem}
\begin{proof}  The proof is essentially identical to the proof of Lemma 2.4 in \cite{JerisonLee:1989}.  \end{proof}

We now use this result to calculate the low order homogeneous terms of the special coframe and the connection $1$-forms.  Recall that for a differential form $\vphi$, 
\begin{equation} \label{eq:ordercalc}
\Lie_P \vphi = P \inmult d\vphi + d(P\inmult \vphi).
\end{equation}  As a result, we have the following proposition.
\begin{prop} \label{prop:coframeconnection}
In pseudohermitian normal coordinates, the low order homogeneous terms of the special coframe and connection $1$-forms are
\begin{multline*}
\eta^i_{(2)}=\frac{1}{2}dt^i -\tensor{I}{\up i \down {\alpha \beta}}x^\alpha dx^\beta; \quad \eta^i_{(3)} = 0; \\
\quad \eta^i_{(m)} = \frac{1}{m} ( t^j \tensor{\omega}{\down j \up i} + \tensor{T}{\up i \down{jk}}t^j \eta^k -2 \tensor{I}{\up i \down{\alpha \beta}}x^\alpha \theta^\beta)_{(m)},\quad m\geq 4; \tag{a}
\end{multline*}
\begin{multline*}
\theta^\alpha_{(1)}=dx^\alpha; \quad \theta^\alpha_{(2)} = 0;\\
 \quad \theta^\alpha_{(m)} = \frac{1}{m}\big( x^\beta \tensor{\omega}{\down \beta \up \alpha} - \tensor{T}{\up \alpha \down{i\gamma}}x^\gamma \eta^i + \tensor{T}{\up \alpha \down{i\beta}}t^i\theta^\beta + \tensor{T}{\up \alpha \down{ij}} t^i \eta^j\big)_{(m)}, \quad m\geq 3;\tag{b}
\end{multline*}
\begin{multline*}
\tensor{\omega}{\down a \up b}_{(1)}=0; \\
 \tensor{\omega}{\down a \up b}_{(m)} = \frac{1}{m}(\tensor{R}{\down{\alpha\beta a}\up b} x^\alpha \theta^\beta + \tensor{R}{\down{\alpha j a}\up b} x^\alpha \eta^j - \tensor{R}{\down{\alpha j a}\up b} t^j \theta^\alpha +  \tensor{R}{\down{ija}\up b}t^i \eta^j)_{(m)}, \\ m\geq 2. \tag{c}
\end{multline*}
\end{prop}
\begin{proof}
The proof is essentially the same as the proof of Proposition 2.5 in \cite{JerisonLee:1989}.
\end{proof}

Let us denote by $\order{m}$ those tensor fields whose Taylor expansions at $q$ contain only terms of order greater than or equal to $m$.  For example, from the above proposition, $\eta^i \in \order{2}$ and $\theta^\alpha \in \order{1}$.  It is routine to check that if $\vphi \in \order{m}$ and $\psi \in \order{m'}$, then $\vphi \otimes \psi \in \order{m+m'}$.  To further extend the utility of this notation we introduce the following:  for any index $a$, let $o(a)=1$ if $a \leq 4n$ and $o(a)=2$ if $a>4n$.  Given a multiindex $A=(a_1, \ldots a_r)$, we let $\#A = r$ and $o(A)=\sum_i o(a_i)$.  Finally, if we have a collection of indexed vector fields, $X_a$, we let $X_A = X_{a_r}\ldots X_{a_1}$, and similarly for similar expressions.

\begin{cor} \label{cor:XT} If we define $X_\alpha = \partial_\alpha+ 2 \tensor{I}{\up i \down{\beta \alpha}}x^\beta \partial_i$ and $T_i = 2 \partial_i$, then $\xi_\alpha = X_\alpha + \order{1}$ and $R_i = T_i+\order{0}$.\end{cor}

This corollary shows that the special frame constructed in these coordinates is particularly close to the standard left-invariant frame on the quaternionic Heisenberg group.  In particular, the vector fields $X_\alpha$ and $T_i$ are the standard left-invariant frame on $\QH^n$ defined in section \ref{sec:quatheis}, and the given frame on $M$ is expressed as a perturbation of them.  Further, as a matter of notation we will occasionally refer to $T_i$ as $X_{4n+i}$, similar to our convention regarding $R_i = \xi_{4n+i}$.

Given any two $Sp(n)Sp(1)$-frames at $q$, they determine distinct parabolic coordinate systems, related by a unique element of $Sp(n)Sp(1)$.  Combining Theorem \ref{thm:coordinates} and Proposition \ref{prop:coframeconnection} we have the following theorem.
\begin{thm} \label{thm:paraboliccoordinates}
Let $M$ be a QC manifold with pseudohermitian structure $\eta$, and let $q\in M$ be any point.  Then there exist parabolic normal coordinates $(x^\alpha,t^i)$ about $q$ for which
\[ g_{\alpha\beta}(q) = \delta_{\alpha\beta},\quad \eta^i(q) = \frac{1}{2}dt^i(q), \quad \tensor{\omega}{\down b \up b}(q) = 0.\]
Further, any two such coordinate systems centered at $q$ are related by a linear transformation in $Sp(n)Sp(1)$.
\end{thm}

We close this section with a lemma on the parabolic version of Taylor expansions.

\begin{lem} \label{lem:parabolic_taylor} Let $F$ be a smooth function defined near $q \in M$.  Then in pseudohermitian normal coordinates, for any nonnegative integer $m$, 
\[ F_{(m)} = \sum_{o(A)=m} \frac{1}{(\#A)!}\Big(\frac{1}{2}\Big)^{o(A)-\# A} x^A (X_A F)|_q.\]
\end{lem}
\begin{proof}
The proof is essentially the same as the proof of Lemma 3.10 in \cite{JerisonLee:1989}. 
\end{proof}

\subsection[QC normal coordinates]{Quaternionic contact  normal coordinates}

Now, using the coordinates constructed above, we will develop a conformal factor $u$ so that the parabolic normal coordinates for the pseudohermitian structure $e^{2u}\eta$ satisfy a number of convenient normalization conditions on the QC curvature and torsion tensors.

We begin with a technical lemma describing the covariant derivative of a tensor field in terms of the action of the vector fields $X_a$ defined in the previous section.
\begin{lem} \label{lem:covD_vect_rel} If $\vphi$ is a tensor in $\order{m}$, the components of its covariant derivatives in terms of a special frame satisfy
\[ \vphi_{A,B} = X_B \vphi_A + \order{m-o(AB)+2}.\]
\end{lem}
\begin{proof}
The proof is essentially the same as the proof of Lemma 3.2 in \cite{JerisonLee:1989}.\end{proof}

\subsubsection{Parabolic coordinates under a conformal change}

Now let us consider the effect of changing the pseudohermitian structure by a conformal factor.  We let $\tilde{\eta}^i = e^{2u}\eta^i$ for some smooth function $u$.  Then $H$ remains the kernel of the three $1$-forms, and it is a simple calculation to see that for $\tilde{I}^i=I^i$ and $\tilde{g}=e^{2u}g$ we have
\[ d\tilde{\eta}^i(X,Y) = 2\tilde{g}(I^iX,Y), \text{ for all } X,Y \in H.\]
As mentioned in Proposition \ref{prop:confchange}, in \cite{Vassilevetal:2006} the authors demonstrate that the Reeb fields for the new structure are given by
\begin{equation} \label{eq:reebchange} \tilde{R}_i = e^{-2u}(R_i -  \tensor{I}{\down i \up \alpha \down \beta} u^\beta \xi_\alpha).\end{equation}
If we define $\tilde{\xi}_\alpha = \xi_\alpha$ and $\tilde{\theta}^\alpha = \theta^\alpha + \tensor{I}{\down i \up \alpha \down \beta}u^\beta \eta^i$, then $\tilde{\theta}^\alpha(\tilde{R}_i)=0$ and $\tilde{\eta}^i(\tilde\xi_\alpha)=0$.

The change of connection $1$-forms under the change of connection is slightly more complicated as shown in the following lemma.
\begin{lem} \label{lem:connchange} Suppose the conformal factor $u$ is order $m\geq 2$ with respect to $P$.  Then the connection $1$-forms of the Biquard connection transform as follows:
\begin{gather*}
  \tensor{\tilde\omega}{\down \alpha \up \beta} = \tensor{\omega}{\down \alpha \up \beta} + \order{m},\\
  \tensor{\tilde\omega}{\down i \up j} = \tensor{\omega}{\down i \up j} +\order{m}.
\end{gather*}
\end{lem}
\begin{proof} From \cite[Prop 3.5]{Vassilevetal:2006} we can calculate the connection $1$-forms on $V$ directly.  In particular, if we write $\grad_H u = u^\alpha \xi_\alpha$, then for $X\in H$,
\begin{align*}
 \tensor{\tilde\omega}{\down i \up j}(X) &= d\tilde{\eta}^j(\tilde{R}_i, X)\\
   &= (2du\wedge \eta^i + d\eta^i)(R_i - I_i \grad_H u, X)\\
   &= 2du\wedge \eta^j(R_i,X) -2du\wedge\eta^j(I_i \grad_H u, X)\\
   &\qquad  + d\eta^j(R_i,X) -d\eta^j(I_i \grad_H u, X)\\
   &= -2\delta_i^j du(X) + d\eta^j(R_i,X) - 2g(I^jI_i \grad_H u, X)\\
   &= \tensor{\omega}{\down i \up j}(X) -2 \delta_i^j du(X) + 2\delta_i^j du(X) - \tensor{\veps}{\up j \down i \up k} g(I_k \grad_H u, X) \\
   &= \tensor{\omega}{\down i \up j}(X) +  \tensor{\varepsilon}{\up j \down {ik}}du(I^kX).
\end{align*}
Since $u\in \order{m}$, so is $du$, and so $ \tensor{\tilde\omega}{\down i \up j}= \tensor{\omega}{\down i \up j}+\order{m}$ acting on $H$.  A similar calculation for the action of $\tensor{\omega}{\down i \up j}$ on $V$ shows that $\tensor{\tilde\omega}{\down i \up j} = \tensor{\omega}{\down i \up j} + \order{m}$.

For the connection $1$-forms in the $H$ directions, we refer to equation $(5.5)$ and the equation immediately following equation  $(5.12)$ in \cite{Vassilevetal:2006}.  They let $S = \tilde{\omega} - \omega$,  $X,Y,Z \in H$, and denote the conformal factor by $\frac{1}{2h}=e^{2u}$.  Then
\begin{multline*}
-2h g(S_X Y, Z) = dh(X) g(Y,Z) - \frac{1}{2}dh(I_i X)d\eta^i(Y,Z) + dh(Y)g(Z,X) \\
 + \frac{1}{2}dh(I_i Y)d\eta^i(Z,X) - dh(Z) g(X,Y) 
+ \frac{1}{2}dh(I_i Z) d\eta^i(X,Y),
\end{multline*}
\begin{multline*}
g(S_{\tilde{R}_i} X,Y) =-\frac{1}{4}\big(\nabla dh(I_iX,Y) -\nabla dh(X,I_iY) - \tensor{\veps}{\down i \up {jk}} \nabla dh(I_jX,I_kY)\big) \\
 - \frac{1}{2h}\big( \tensor{\veps}{\down i \up {jk}} dh(I_kX)dh(I_j Y)+ dh(I_iX)dh(Y) - dh(I_iY)dh(X)\big) \\
 + \frac{1}{4n}\Big( -\laplace h + \frac{2}{h} \big|dh|_H\big|^2\Big)g(I_iX,Y) - \tensor{\veps}{\down i \up {jk}} dh(R_k)g(I_jX,Y).
\end{multline*}
Here we see that $\tensor{\tilde\omega}{\down \alpha \up \beta}$ and $\tensor{\omega}{\down \alpha \up \beta}$ differ by terms involving $dh$ and $\nabla^2 h|_H$.  The relation between $u$ and $h$ implies that $dh = -e^{-2u}du$ and \[\nabla^2 h=2e^{-2u}du\otimes du - e^{-2u}\nabla^2 u.\]  Since $u$ is order $m$, so is $du$ and from Lemma \ref{lem:covD_vect_rel}
\begin{align*}
 \nabla^2 u|_H  &= u_{\alpha \beta}\, \theta^\alpha \otimes \theta^\beta\\
   &= X_{\beta}X_{\alpha}u\,\theta^\alpha \otimes \theta^\beta +\order{m+2}
\end{align*}
which is also order $m$.  The last term, $dh(R_k)g(I_jX,Y)$ is also order $m$ since $dh(R_k)\in\order{m-2}$ and $g\in\order{2}$.
\end{proof}

Now we are in a position to relate the covariant derivative of the tilded connection to that of the untilded connection.  This will allow us to work only with the original connection by accounting for the orders of the error terms.  We have the following lemma.
\begin{lem} \label{lem:covderchange} Let $\vphi$ be an $s$-tensor and denote by $\nabla^r \vphi$ and $\tilde\nabla^r\vphi$ its $r$th covariant derivatives with respect to the original and rescaled connections respectively.  Let $A$ and $B$ be multiindices with $\#A=s$ and $\#B=r$, and let $\vphi_{A,B}$ and $\tilde\vphi_{A,B}$ denote the components of $\nabla^r \vphi$ and $\tilde\nabla^r \vphi$.  For a conformal change as described above with $u\in \order{m}$, $m\geq2$, we have
\[ \tilde\vphi_{A,B} = \vphi_{A,B} + \order{m-o(B)-1}.\]
Further, if $o(A)=s$ (i.e. $A$ contains no entries greater than $4n$) then
\[ \tilde{\vphi}_{A,B} = \vphi_{A,B} + \order{m-o(B)}.\]
\end{lem}
\begin{proof}
The proof is essentially the same as the proof of Lemma 3.5 in \cite{JerisonLee:1989}.\end{proof}

Finally, we will need to know how two sets of parabolic normal coordinates are related for conformally related pseudohermitian structures.  This is the content of the next lemma, which also corrects an error in the 1989 paper of Jerison and Lee \cite{JerisonLee:1989}.

\begin{thm} \label{thm:coordchange} Let $\Psi$ and $\tilde\Psi$ denote the parabolic exponential maps based at $q\in M$ of the pseudohermitian structures $\eta$ and $\tilde\eta=e^{2u}\eta$, respectively.  Suppose that $u \in \order{m}$ with $m\geq 2$.  Then considered as functions on $T_qM$ with the induced pseudohermitian structure, $\tilde\Psi - \Psi$ is order $m+1$.
\end{thm}
\begin{proof}
We will work in parabolic normal coordinates on $M$ given by the original pseudohermitian structure., written as always as $(x^\alpha, t^i)$.  Identifying a neighborhood of $0\in T_qM$ with a neighborhood of $q\in M$, we may write $\Psi(x,t)=(x,t)$.  Then writing $\tilde\Psi^a(x,t)=x^a + f^a(x,t)$ we need only show that $f^a$ is order $m+1$ for each $a=1,\ldots, 4n+3$.  Since these are parabolic normal coordinates defined by the parabolic geodesics of Theorem \ref{thm:coordinates}, this is equivalent to showing the for any particular $(x,t)$, $f^a(sx, s^2t) = O(s^{m+1})$ as $s\to0$.  Further, if we write $\gamma$ and $\tilde\gamma$ for the parabolic geodesics with initial data $(X,R)=(x^\alpha,t^i)$ at $q$ for the original and rescaled connections, we are reduced to showing that $\tilde\gamma(s) - \gamma(s) \in  O(s^{m+1})$ for small $s$.

Now, let us denote by $\Gamma_{bc}^a$ and $\tilde\Gamma_{bc}^a$ the Christoffel symbols of the two connections in these coordinates, and write $B_{bc}^a=\tilde\Gamma_{bc}^a - \Gamma_{bc}^a$ for the difference tensor.  From equation \eqref{eq:curveeq} we see that $\sigma(s) = \tilde\gamma(s)-\gamma(s)$ satisfies the third order equation
\begin{multline*}
 \dddot{\sigma}^a(s) = \big( \ddot{\gamma}^b(s)\dot\gamma^c(s)\Gamma_{bc}^a(\gamma(s)) - \ddot{\tilde\gamma}^b(s)\dot{\tilde\gamma}^c(s)\tilde\Gamma_{bc}^a(\tilde\gamma(s))\big)\\
  +2\big( \dot\gamma^b(s) \ddot\gamma^c(s)\Gamma_{bc}^a(\gamma(s)) - \dot{\tilde\gamma}^b(s) \ddot{\tilde\gamma}^c(s)\tilde\Gamma_{bc}^a(\tilde\gamma(s))\big)\\
  + \big(\dot\gamma^b(s)\dot\gamma^c(s)\dot\gamma^d(s)\partial_d \Gamma_{bc}^a(\gamma(s)) - \dot{\tilde\gamma}^b(s)\dot{\tilde\gamma}^c(s)\dot{\tilde\gamma}^d(s)\partial_d \tilde\Gamma_{bc}^a(\tilde\gamma(s))\big)\\
  + \big( \dot\gamma^b(s)\dot\gamma^d(s)\dot\gamma^e(s)\Gamma_{de}^c(\gamma(s))\Gamma_{bc}^a(\gamma(s)) \\- \dot{\tilde\gamma}^b(s)\dot{\tilde\gamma}^d(s)\dot{\tilde\gamma}^e(s)\tilde\Gamma_{de}^c(\tilde\gamma(s))\tilde\Gamma_{bc}^a(\tilde\gamma(s))\big),
\end{multline*}
with initial conditions $\sigma^a(0)=0$, $\dot\sigma^a(0)=0$ and $\ddot\sigma^a(0)=-B_{\beta\gamma}^a(0) x^\beta x^\gamma$.  From Lemma \ref{lem:connchange} we see that $B_{\beta\gamma}^a$ is order $m-1$ and so $\ddot\sigma(0)^a=0$ as well.

Let us simplify the notation by omitting the dependence on $s$.  To that end we write $\Gamma_{bc}^a=\Gamma_{bc}^a(\gamma(s))$, $\tilde\Gamma_{bc}^a = \tilde\Gamma(s)_{bc}^a(\tilde\gamma(s))$ and $\hat\Gamma_{bc}^a = \tilde\Gamma_{bc}^a(\gamma(s))$. Then the equation above becomes much more compact:
\begin{multline} \label{eq:sigmadiffeq}
 \dddot\sigma^a = (\ddot\gamma^b\dot\gamma^c\Gamma_{bc}^a - \ddot{\tilde\gamma}^b\dot{\tilde\gamma}^c\tilde\Gamma_{bc}^a) + 2(\dot\gamma^b \ddot\gamma^c \Gamma_{bc}^a - \dot{\tilde\gamma}^b \ddot{\tilde\gamma}^c \tilde\Gamma_{bc}^a) \\
 + (\dot\gamma^b\dot\gamma^c \dot\Gamma_{bc}^a - \dot{\tilde\gamma}^b\dot{\tilde\gamma}^c \dot{\tilde\gamma}_{bc}^a) + (\dot{\tilde\gamma}^b\dot{\tilde\gamma}^d\dot{\tilde\gamma}^e\tilde\Gamma_{de}^c\tilde\Gamma_{bc}^a - \dot{\tilde\gamma}^b\dot{\tilde\gamma}^d\dot{\tilde\gamma}^e\tilde\Gamma_{de}^c\tilde\Gamma_{bc}^a)
\end{multline}

Our goal is to estimate $\dddot\sigma^a$ and then derive bounds on it to prove the theorem.  Thus we shall expand the right-hand side of equation \eqref{eq:sigmadiffeq}.  We present one example of this expansion and leave it to the reader to complete the rest.
\begin{align*}
 \ddot\gamma^b\dot\gamma^c\Gamma_{bc}^a - \ddot{\tilde\gamma}^b\dot{\tilde\gamma}^c\tilde\Gamma_{bc}^a &= (\ddot\gamma^b - \ddot{\tilde\gamma}^b)\dot{\tilde\gamma}^c\tilde\Gamma_{bc}^a + \ddot\gamma^b(\dot\gamma^c-\dot{\tilde\gamma}^c)\tilde\Gamma_{bc}^a \\
   & \quad + \ddot\gamma^b\dot\gamma^c(\hat\Gamma_{bc}^a - \tilde\Gamma_{bc}^a) + \ddot\gamma^b \dot\gamma^c (\Gamma_{bc}^a - \hat\Gamma_{bc}^a)\\
    &= -\ddot\sigma^b\dot{\tilde\gamma}^c\tilde\Gamma_{bc}^a - \ddot\gamma^b\dot\sigma^c\tilde\Gamma_{bc}^a 
+ \ddot\gamma^b\dot\gamma^c(\hat\Gamma_{bc}^a - \tilde\Gamma_{bc}^a) - \ddot\gamma^b \dot\gamma^c B_{bc}^a.
\end{align*}
Using this technique we have the following bound for $|\dddot\sigma^a|$,
\begin{multline*}
|\dddot\sigma^a| \leq  |\ddot\sigma^b\dot{\tilde\gamma}^c\tilde\Gamma_{bc}^a| + |\ddot\gamma^b\dot\sigma^c\tilde\Gamma_{bc}^a| + |\ddot\gamma^b\dot\gamma^c(\hat\Gamma_{bc}^a - \tilde\Gamma_{bc}^a)| + |\ddot\gamma^b \dot\gamma^c B_{bc}^a| \\
  +2|\dot\sigma^b\ddot{\tilde\gamma}^c\tilde\Gamma_{bc}^a| + 2|\dot\gamma^b\ddot\sigma^c\tilde\Gamma_{bc}^a| + 2|\dot\gamma^b\ddot\gamma^c(\hat\Gamma_{bc}^a - \tilde\Gamma_{bc}^a)| + 2|\dot\gamma^b \ddot\gamma^c B_{bc}^a| \\
  +|\dot\sigma^b\dot{\tilde\gamma}^c\dot{\tilde\Gamma}_{bc}^a| + |\dot\gamma^b\dot\sigma^c\dot{\tilde\Gamma}_{bc}^a| + |\dot\gamma^b\dot\gamma^c(\dot{\hat\Gamma}_{bc}^a - \dot{\tilde\Gamma}_{bc}^a)| + |\dot\gamma^b \dot\gamma^c \dot{B}_{bc}^a| \\
  +|\dot\sigma^b\dot{\tilde\gamma}^d\dot{\tilde\gamma}^e\tilde\Gamma_{de}^c\tilde\Gamma_{bc}^a| + |\dot\gamma^b\dot\sigma^d\dot{\tilde\gamma}^e\tilde\Gamma_{de}^c\tilde\Gamma_{bc}^a|+ |\dot\gamma^b\dot\gamma^d\dot\sigma^e\tilde\Gamma_{de}^c\tilde\Gamma_{bc}^a| \\
  + |\dot\gamma^b\dot\gamma^d\dot\gamma^e B_{de}^c\tilde\Gamma_{bc}^a| + |\dot\gamma^b\dot\gamma^d\dot\gamma^e (\hat\Gamma_{de}^c-\tilde\Gamma_{de}^c)\tilde\Gamma_{bc}^a| \\
  + |\dot\gamma^b\dot\gamma^d\dot\gamma^e \Gamma_{de}^cB_{bc}^a| + |\dot\gamma^b\dot\gamma^d\dot\gamma^e \Gamma_{de}^c(\hat\Gamma_{bc}^a-\tilde\Gamma_{bc}^a)|.
\end{multline*}

Since each of $\gamma^a$, $\tilde\gamma^a$, $\Gamma_{bc}^a$ and $\tilde\Gamma_{bc}^a$ is a smooth function, by further shrinking the neighborhood of $q$ we are considering we may bound each of these functions and their derivatives by a uniform constant.  Further, from Lemma \ref{lem:connchange} and $B_{bc}^a = dx^a(\tilde\nabla_{\partial_b}\partial_c - \nabla_{\partial_b}\partial_c)$, we have $ B_{bc}^a \in\order{m-o(b)}$.
Since \mbox{$\dot\gamma^b(s) = O(s^{o(b)-1})$}, we therefore have
\[ B_{bc}^a\dot\gamma^b \in\order{m-1}, \text{ and } B_{bc}^a \ddot\gamma^b,\ \dot{B}_{bc}^a\dot\gamma^b \in \order{m-2}.\]
Finally, since $\tilde\Gamma_{bc}^a$ is a smooth function, it satisfies a Lipschitz estimate
\[ |\hat\Gamma_{bc}^a(s) - \tilde\Gamma_{bc}^a(s)| \leq C |\gamma(s)-\tilde\gamma(s)| \leq C \sum_b |\sigma^b(s)|.\]

Now we define $\vphi(s) = \sum_a \big( |\ddot\sigma^a(s)|^2 + |\dot\sigma^a(s)|^2 + |\sigma^a(s)|^2\big)$, so that
\begin{align*}
 |\dddot\sigma^a(s)| &\leq C\Big( \sum_b \big( |\ddot\sigma^b(s)| +|\dot\sigma^b(s)| + |\sigma^b(s)|\big) + s^{m-2} \Big) \\
 &\leq C( \vphi(s)^{1/2} + s^{m-2}).
\end{align*}
Taking the derivative of $\vphi$ we find
\begin{align*}
 | \dot\vphi(s)| &= 2 \Big|\sum_b (\dddot\sigma^b(s)\ddot\sigma^b(s) + \ddot\sigma^b(s) \dot\sigma^b(s) + \dot\sigma^b(s) \sigma^b(s)) \Big| \\
  & \leq C( \vphi(s) + \vphi(s)^{1/2}s^{m-2}).
\end{align*}
 
It is simple to check that the ODE $ \dot{y}(s) = C(y(s) + y(s)^{1/2}s^{m-2})$ with initial condition $y(0)=0$ has a family of solutions given by \[y_a(s)= \begin{cases}0,&s\leq a \\ \frac{C^2}{4}e^{Cs}\beta^2_{m-2}(s;a),& s\geq a\end{cases}\]  where $\beta_k(s;a) = \int_a^s e^{-Ct/2} t^k\, dt$.   A routine calculation shows that \[\beta_k(s;a) = O((s-a)^{k+1}),\] so $y_0(s) = O(s^{2m-2})$.  Further, $y_0(s)\geq y_a(s)$ for all $a\geq 0$ and hence by Theorem \ref{thm:hartman} below, $\vphi(s) = O(s^{2m-2})$, which implies $\sigma(s) = O(s^{m+1})$.  This completes the proof.
\end{proof}

Theorem \ref{thm:coordchange} above relies on a technical comparison theorem for ODEs given in \cite[Theorem III.4.1]{Hartman:1973}.  The paper which inspired this work, \cite{JerisonLee:1989}, fails to recognize the infinite family of solutions to the ODE $ \dot{y}(s) = C(y(s) + y(s)^{1/2}s^{m-2})$, and so their proof is incorrect as it stands.  The argument given above also completes their proof, with the following theorem.
\begin{thm}[\cite{Hartman:1973}] \label{thm:hartman}
Let $U(t,u)$ be continuous on an open $(t,u)$-set $E$ and $u=u^0(t)$ the maximal solution of
\[ \dot{u}=U(t,u), \quad u(t_0)=u_0.\]
Let $v(t)$ be a continuous function on $[t_0,t_0+a]$ satisfying the conditions $v(t_0)\leq u_0$, $(t,v(t)) \in E$, and $v(t)$ has a right derivative $D_R v(t)$ on $t_0 \leq t \leq t_0+a$ such that
\[ D_R v(t) \leq U(t,v(t)).\]
Then, on a common interval of existence of $u^0(t)$ and $v(t)$, 
\[ v(t) \leq u^0(t).\]
\end{thm}

\subsubsection{Curvature and torsion normalizations} \label{sec:curvandtornormalizations}

Now we turn to the coordinate normalization that is the focus of this paper.  For each QC pseudohermitian structure, we will construct a $2$-tensor $Q$, defined in such a way that by considering the tensors $Q$ and $\tilde Q$ determined by a conformal change, we may recover the symmetric covariant Hessian of the conformal factor.  Since the antisymmetric covariant Hessian is determined by the first derivatives and torsion, this completely determines the Hessian of the conformal factor.  In this section we use the common notation $F_{(ab)} = \frac{1}{2}(F_{ab} + F_{ba})$ for the symmetric part of the tensor $F_{ab}$.  More generally, $F_{(A)} = \frac{1}{(\#A)!}\sum_{\sigma \in S_{\#A}} F_{\sigma A}$, where the sum is over the permutation group on $\#A$ letters, and $\sigma \in S_{\#A}$ acts on the multiindex $A$ by permuting the indices.  That is, $F_{(A)}={Sym}(F)_A$, where $
{Sym}(F)$ is the symmetric part of $F$.

Let $u\in\order{m}$ be a fixed conformal factor.
From the transformation rules for $\tau_{\alpha\beta}$, $\mu_{\alpha\beta}$ and $S$ in Proposition \ref{prop:confchange} we know that the tensor $L_{\alpha\beta}= \frac{1}{2} \tau_{\alpha\beta} + \mu_{\alpha\beta} + \frac{S}{32n(n+2)}g_{\alpha\beta}$ transforms as
\begin{equation}\label{eq:Lchange} \tilde{L}_{\alpha\beta} = L_{\alpha\beta} - u_{(\alpha\beta)} + \order{m-1},\end{equation}
and a routine calculation shows the torsion tensor $\tensor{T}{\up \alpha \down{ij}} = d\theta^\alpha(R_i,R_j)$ changes as
\[ \tensor{\tilde T}{\up \alpha \down {ij}} = \tensor{T}{\up \alpha \down {ij}} + (\tensor{I}{\down j \up \alpha \down \beta} \tensor{u}{\up \beta \down i} - \tensor{I}{\down i \up \alpha \down \beta}\tensor{u}{\up \beta \down j}) + \order{m-2}.\]
From this and the fact that the volume form on $V$ provides an isomorphism between $V$ and $\bigwedge^2 V$, 
\begin{equation}\label{eq:torsionchange} \tensor{\tilde T}{\down {\alpha jk}}\tensor{\tilde\veps}{\down i \up {jk}} = \tensor{T}{ \down {\alpha jk}}\tensor{\veps}{\down i \up{jk}} + \tensor{A}{\down {i\alpha} \up{j\beta}}u_{\beta j}+ \order{m-2},\end{equation}
where $\tensor{A}{\down{i\alpha} \up {j\beta}}=2\tensor{\veps}{\up {jk} \down i} \tensor{I}{\down {k\alpha}\up \beta}$.  The operator $A$  is invertible because its minimal polynomial is $m_A(s) = s^2+2s-8$, which follows from
\begin{align*}
\tensor{A}{\down{i\alpha}\up{j\beta}}\tensor{A}{\down{j\beta}\up{k\gamma}} &= 4 \tensor{\veps}{\up{jl}\down i} \tensor{\veps}{\up {km}\down j} \tensor{I}{\down {l\alpha}\up\beta}\tensor{I}{\down{m\beta}\up \gamma} \\
	&= 4(\delta^{kl}\delta_i^m - \delta^{lm}\delta_i^k)(\tensor{\veps}{\down {lm}\up p} \tensor{I}{\down {p\alpha}\up \gamma} - \delta_{lm} \delta_\alpha^\gamma)\\
	&=4(\tensor{\veps}{\up k\down {i}\up p} \tensor{I}{\down {p\alpha}\up \gamma} - \delta_{i}^k \delta_\alpha^\gamma) + 4(3\delta_i^k \delta_\alpha^\gamma)\\
	&= -2\tensor{A}{\down {i\alpha} \up {k\gamma}} +8 \delta_i^k\delta_\alpha^\gamma.
\end{align*}
Because $A$ is constructed from the metric tensors on $H$ and $V$, the almost complex structures and the $V$-volume form, we have $\tilde{A} = A + \order{m}$ and $\tilde{A}^{-1} = A^{-1} + \order{m}$.

Let us now define the tensor 
\begin{equation} \label{eq:defineB}
B_{ij}  = R_{kl\alpha\beta}\tensor{\veps}{\up {kl} \down i} \tensor{I}{\down j \up {\alpha\beta}}.
\end{equation}
Since the connection preserves the decomposition of the tangent bundle, we know that $R_{\alpha i j \beta} = R_{j\alpha i\beta}=0$, and hence the first Bianchi identity \cite[Equations (3.1) and (3.2)]{Vassilevetal:2007} shows
\begin{equation}  R_{kl\alpha\beta} = T_{\beta l\gamma}\tensor{T}{\up\gamma\down{k\alpha}} - T_{\beta k \gamma}\tensor{T}{\up \gamma \down {l\alpha}} + T_{\beta m \alpha}\tensor{T}{\up m\down {kl}} +T_{\beta kl,\alpha}  + T_{\beta l \alpha,k} -T_{\beta k \alpha,l}.\end{equation}
Since the curvature is tensorial, by working at a point where the frame is parallel so that $\nabla I_i=0$, we may simplify this using Propositions \ref{prop:torsionprop} and \ref{prop:curvprop} as
\begin{multline}\label{eq:firstbianchi}
R_{kl\alpha\beta} = T_{\beta l\gamma}\tensor{T}{\up\gamma\down{k\alpha}} - T_{\beta k \gamma}\tensor{T}{\up \gamma \down {l\alpha}}  + T_{\beta kl,\alpha}+ (\mu_{\beta\gamma,k} \tensor{I}{\down l \up {\gamma} \down \alpha} - \mu_{\beta\gamma,l} \tensor{I}{\down k \up {\gamma} \down \alpha} )  \\
 - \frac{1}{8n(n+2)}S\, \veps_{klm}\Big( \mu_{\beta\gamma} \tensor{I}{\up {m\gamma} \down \alpha} + \frac{1}{4}(\tau_{\beta\gamma} \tensor{I}{\up {m\gamma} \down \alpha} - \tau_{\alpha\gamma}\tensor{I}{\up {m\gamma} \down \beta})\Big)  \\
 + \frac{1}{4}(\tau_{\alpha \gamma,l} \tensor{I}{\down k \up \gamma \down \beta} - \tau_{\alpha \gamma,k} \tensor{I}{\down l \up \gamma \down \beta} - \tau_{\beta \gamma,l} \tensor{I}{\down k \up \gamma \down \alpha} + \tau_{\beta \gamma,k} \tensor{I}{\down l \up \gamma \down \alpha}).
\end{multline}
Because $\mu_{\alpha\beta}$ and $\tau_{\alpha\beta}$ are symmetric and trace-free, contracting this with an almost complex structure on the horizontal indices yields
\begin{equation}\label{eq:firstbianchi2} 
R_{kl\alpha\beta}\tensor{I}{\down j \up {\alpha\beta}} = T_{\beta kl,\alpha}\tensor{I}{\down j \up {\alpha\beta}} + (T_{\beta l\gamma}\tensor{T}{\up\gamma\down{k\alpha}} - T_{\beta k \gamma}\tensor{T}{\up \gamma \down {l\alpha}}) \tensor{I}{\down j \up {\alpha\beta}}.
\end{equation}
We know from Proposition \ref{prop:confchange} that the terms in parentheses change under a conformal rescaling by terms of order at least $m-2$, which we will be able to ignore below.  Thus
\[ \tilde{B}_{ij} - B_{ij} = \tilde{T}_{\beta kl,\alpha}\tensor{I}{\down j \up {\alpha\beta}}\tensor{\tilde{\veps}}{\up {kl} \down i} - T_{\beta kl,\alpha}\tensor{I}{\down j \up {\alpha\beta}}\tensor{{\veps}}{\up {kl} \down i} + \order{m-2}.\]
Notice that we have already calculated $\tilde{T}_{\beta kl} - T_{\beta kl}$ above and seen that it depended on a second derivative of $u$, one derivative each in the vertical and horizontal directions.  Since we are now taking another derivative and $\tensor{I}{\down j \up{\alpha\beta}}$ is antisymmetric in the horizontal indices, we expect that the contraction should result in only a second covariant derivative of $u$ in the vertical direction.  Modulo terms of order $m-3$, we have
\begin{align*}
\tilde{T}_{\beta kl,\alpha}\tensor{I}{\down j \up {\alpha\beta}}\tensor{\tilde{\veps}}{\up {kl} \down i} - T_{\beta kl,\alpha}\tensor{I}{\down j \up {\alpha\beta}}\tensor{{\veps}}{\up {kl} \down i} &= (\tilde{T}_{\beta kl}\tensor{\tilde{\veps}}{\up {kl} \down i} - T_{\beta kl}\tensor{{\veps}}{\up {kl} \down i})_{,\alpha}\tensor{I}{\down j \up {\alpha\beta}} + \ldots \\
&= 2 \tensor{\veps}{\up{kl}\down i} \tensor{I}{\down {l\beta} \up \gamma}  \tensor{I}{\down j \up {\alpha\beta}}u_{k\gamma\alpha} + \ldots \\
&= 4 \tensor{I}{\down i \up {\alpha\gamma}} u_{j\gamma\alpha} + \ldots \\
&= 4 \tensor{I}{\down i \up {\alpha\gamma}} u_{\gamma\alpha j} + \ldots \\
&= (16 n) u_{ij} + \ldots  .
\end{align*}
Here we are making use of the fact that commuting vertical covariant derivatives with horizontal ones depends only on terms of order $m-2$ or greater and that \[u_{\gamma\alpha j} = u_{\alpha\gamma j} + 2 \tensor{I}{\up k \down {\alpha\gamma}} u_{kj} + \order{m-3}.\]

This implies that under a conformal change, the symmetric part of the tensor $B$ transforms as
\begin{equation}\label{eq:Bchange}
\tilde{B}_{(ij)} = B_{(ij)} + (16n) u_{(ij)} + \order{m-3}.
\end{equation}

With these identities in mind, we define the tensor $Q$ as
\begin{gather*}
 Q_{\alpha \beta} =  L_{\alpha\beta} +\frac{1}{8(n+2)} S g_{\alpha \beta},\\
 Q_{\alpha i} = Q_{i \alpha}= -\tensor{(A^{-1})}{\down {i\alpha}\up{j\beta}} T_{\beta kl}\tensor{\veps}{\down j \up{kl}},\\
 Q_{ij} = -\frac{1}{16n}B_{(ij)}.
\end{gather*}
Then according to \eqref{eq:scalarchange}, \eqref{eq:Lchange}, \eqref{eq:torsionchange} and \eqref{eq:Bchange}, under a conformal change with $u\in \order{m}$, $Q$ changes as
\begin{gather*}
\tilde{Q}_{\alpha\beta} - Q_{\alpha\beta} = -u_{(\alpha\beta)} + (\laplace_H u)g_{\alpha\beta} + \order{m-1}\\
\tilde{Q}_{i\alpha} - Q_{i\alpha} = -u_{i \alpha} + \order{m-2} \\
\tilde{Q}_{ij} - Q_{ij} = -u_{(ij)} + \order{m-3}.
\end{gather*}

Now, from section \ref{sec:parabolictaylor}, we recall that the vector field $P= x^a \xi_a$ is the generator of the parabolic dilations that inspired the parabolic normal coordinate construction.  Under the conformal change, by Theorem \ref{thm:coordchange} the new coordinates satisfy $\tilde{x}^a = x^a + \order{m+1}$, and hence $\tilde P = P + \order{m-1}$, since $P$ depends on both $\xi_\alpha \in \order{-1}$ and $R_i\in \order{-2}$.  To study the Taylor expansion of $\tilde{Q}$ we define the scalar $\Phi = Q(P,P)= x^a x^b Q_{ab}$.  Then by the above comments and the fact that a smooth $2$-tensor is at least order $2$,  we have
\[ \tilde \Phi = \tilde{Q}(\tilde{P}, \tilde{P}) = \tilde{Q}(P, P) + \order{m+1} = x^a x^b \tilde{Q}_{ab} + \order{m+1}.\]

Recall that the vector fields $X_\alpha$ and $T_i$ defined in Corollary \ref{cor:XT} are the standard left invariant frame on the quaternionic Heisenberg group, and covariant differentiation using the Biquard connection is represented as a perturbation of the derivatives with respect to $X_\alpha$ and $T_i$ according to Lemma \ref{lem:covD_vect_rel}.  We let $\Hlaplace = -\sum_\alpha X_\alpha X_\alpha$ denote the standard sublaplacian on $\QH^n$.  Combining the above calculations with our definitions of $\Phi$ and $\tilde \Phi$, we have
\begin{align*}
\tilde \Phi &= \Phi - (u_{(\alpha\beta)}x^\alpha x^\beta +  2u_{i\alpha}t^i x^\alpha +  u_{(ij)}t^i t^j) + \laplace_H u g_{\alpha \beta} x^\alpha x^\beta +\order{m+1}\\
 &= \Phi - (x^\alpha x^\beta X_\beta X_\alpha u +  2t^i x^\alpha X_\alpha T_i u + t^i t^j T_j T_i u) + |x|^2\Hlaplace u +\order{m+1}
\end{align*}

Now we let $\PP_m$ denote the space of homogeneous polynomials in $x$ and $t$ of order $m$.  For $u\in \PP_m$, it is immediate that 
\begin{equation*}\tilde{\Phi}_{(m)} = \Phi_{(m)} -  (x^\alpha x^\beta X_\beta X_\alpha u + 2 t^i x^\alpha X_\alpha T_i u +  t^i t^j T_j T_i u) + |x|^2\Hlaplace u .\end{equation*}
Further, 
\begin{align*}
m^2u = P^2 u &= (x^\alpha X_\alpha + t^i T_i)^2 u \\
  &= x^\alpha X_\alpha u+ x^\alpha x^\beta X_\alpha X_\beta u+ 2t^i x^\alpha X_\alpha T_i u+ 2t^i T_i u + t^i t^j T_i T_ju \\
  &= x^\alpha x^\beta X_\beta X_\alpha u + 2t^i x^\alpha X_\alpha  T_i u + t^i t^j T_j T_i u + t^i T_i u + Pu\\
  &= x^\alpha x^\beta X_\beta X_\alpha u + 2 t^i x^\alpha  X_\alpha  T_i u + t^i t^j T_j T_i u + t^i T_i u + mu.
\end{align*}
Combining this with the above calculation for $\tilde{\Phi}_{(m)}$, we have
\begin{equation}\label{eq:Phichange} \tilde{\Phi}_{(m)}= \Phi_{(m)} - m(m-1)u + t^i T_i u + |x|^2 \Hlaplace u,\end{equation}
whenever $u \in \PP_m$.

\begin{lem}\label{lem:invertop} The operator $L_m = |x|^2 \Hlaplace  + t^i T_i  - m(m-1)$ is invertible on $\PP_m$ for $m\geq 3$.  For $m=2$, $L_2$ has kernel the subspace of $\PP_2$ spanned by $t^i$, $i=1,2,3$, and is invertible on the subspace depending only on $x^\alpha$.
\end{lem}
\begin{proof}
The proof is essentially the same as the proof of Lemma 3.9 in \cite{JerisonLee:1989}.
\end{proof}

Using Lemma \ref{lem:parabolic_taylor} we may write
\[ \tilde{\Phi}_{(m)} = \sum_{o(abC)=m} \frac{1}{(\# C)!} \Big(\frac{1}{2}\Big)^{o(C)-\# C} x^a x^b x^C X_C \tilde{Q}_{ab}|_q\]
and
\[ \Phi_{(m)} = \sum_{o(abC)=m} \frac{1}{(\# C)!} \Big(\frac{1}{2}\Big)^{o(C)-\# C} x^a x^b x^C X_C Q_{ab}|_q.\]
Now if we choose $u\in \PP_m$, then by Lemma \ref{lem:covderchange}, the covariant derivatives $\tilde{Q}_{ab,C}$ with \mbox{$o(abC)=m$} may be computed with respect to the original connection with an error of order \mbox{$m-o(C)-1$}, which vanishes at $q$.  Furthermore, by our calculations above, $\tilde{Q}-Q \in \order{m}$ and so by Lemma \ref{lem:covD_vect_rel}, at $q$,
\[ \tilde{Q}_{ab,C} - Q_{ab,C} = X_C \tilde{Q}_{ab} - X_C Q_{ab}.\]
Thus
\begin{multline} \label{eq:u_eq} L_m u= \tilde{\Phi}_{(m)} - \Phi_{(m)} \\= \sum_{o(abC)=m} \frac{1}{(\# C)!} \Big(\frac{1}{2}\Big)^{o(C)-\# C}x^ax^b x^C \big(\tilde{Q}_{ab,C}(q) - Q_{ab,C}(q)\big).
\end{multline}

The next lemma is the key ingredient in showing that we may force symmetrized covariant derivatives of $Q$ to vanish by appropriately choosing $u$.
\begin{lem} \label{lem:order_m_vanish}Let $q\in M$ and $(x^\alpha, t^i)$ be pseudohermitian normal coordinates centered at $q$ for a pseudohermitian structure $\eta$.  For any $m\geq 2$, there is a polynomial $u \in \PP_m$ in the coordinates $(x,t)$ such that $\tilde{\eta} = e^{2u}\eta$ satisfies
\[ \tilde{Q}_{(ab,C)}(q) = 0 \text{ if } o(abC)=m.\]
For $m\geq 3$ the polynomial is unique, while for $m=2$, it is unique in $\mathcal{R}_2$.
\end{lem}
\begin{proof}
By Lemma \ref{lem:invertop}, if $m\geq 3$ there is a unique polynomial $u\in \PP_m$ such that 
\[ L_m u = -\sum_{o(abC)=m} \frac{1}{(\# C)!} \Big(\frac{1}{2}\Big)^{o(C)-\# C} x^a x^b x^C Q_{ab,C}|_q.\]
For $m=2$, the right hand side is independent of $t$, and so there is a unique polynomial in $\mathcal{R}_2$.  If we now set $\tilde{\eta} = e^{2u}\eta$, it follows from \eqref{eq:u_eq} that 
\[ \sum_{o(abC)=m} \frac{1}{(\# C)!} \Big(\frac{1}{2}\Big)^{o(C)-\# C} x^a x^b x^C \tilde{Q}_{ab,C}|_q = 0.\]
For any multiindex $abC$, the coefficient of $x^ax^bx^C$ is a nonzero multiple of $\tilde{Q}_{(ab,C)}$.  Thus we have determined the required polynomial.
\end{proof}

Finally, we come to the proof that symmetrized covariant derivatives of $Q$ can be made to vanish.
\begin{thm}[Main Theorem] \label{thm:symder_vanish}  Let $M$ be a QC manifold.  For any $q\in M$ and any $N\geq 2$, there is a choice of pseudohermitian structure $\eta$ such that all the symmetrized covariant derivatives of $Q$ with total order less than or equal to $N$ vanish at $q$; that is,
\[ Q_{(ab,C)}(q) = 0 \text{ if } o(abC)\leq N.\]
If we write $\eta = e^{2u} \tilde{\eta}$ for another pseudohermitian structure, we may arbitrarily choose the $1$-jet of $u$ at $q$.  Once this is fixed, the Taylor series of $u$ at $q$ is uniquely determined.
\end{thm}
\begin{proof}
We apply Lemma \ref{lem:order_m_vanish} repeatedly for $2\leq m \leq N$.  This works because using $u\in \PP_m$ as a conformal factor does not change terms of the form $Q_{ab,C}$ with $o(abC)<m$.  Choosing the $1$-jet of $u$ allows us to inductively determine higher order parts of the Taylor series.
\end{proof}

Using the above normalization for $Q$ and a host of identities from \cite{Vassilevetal:2006} and \cite{Vassilevetal:2007} we show that at the center point $q$, the Ricci tensor, scalar curvature, quaternionic contact torsion and many of their covariant derivatives vanish.
\begin{thm} \label{thm:mainthm} Let $M$ be a QC manifold and $\eta$ a pseudohermitian structure for which the symmetrized covariant derivatives of the tensor $Q$ vanish to total order $4$ at a point $q$.  The the following curvature and torsion terms vanish at $q$.
\begin{gather*}
S,\ \tau_{\alpha\beta},\ \mu_{\alpha\beta},\ L_{\alpha\beta},\ R_{\alpha\beta},\ T_{\alpha i \beta},\ T_{ijk} \\
T_{\alpha jk},\ S_{,\beta},\ \tensor{\mu}{\down{\alpha\beta,}\up\alpha},\ \tensor{\tau}{\down {\alpha\beta,} \up \alpha} \\
B_{ij},\ S_{,i},\ \tensor{S}{\down {,\alpha} \up \alpha},\ \tensor{\tau}{\down{\alpha\beta,}\up{\alpha\beta}},\ \tensor{\mu}{\down{\alpha\beta,}\up{\alpha\beta}},\ \tensor{R}{\down{\gamma i \beta}\up\gamma\down {, \alpha}}I^{i\beta\alpha}.
\end{gather*}
\end{thm}
\begin{proof}
Considering first the terms of $Q$ with multiindex of order $2$, we take the horizontal trace to find that at $q$,
\[ \tensor{Q}{\down \alpha \up \alpha} = \frac{4n+1}{8(n+2)} S = 0,\]
form which it is clear that the pseudohermitian scalar curvature vanishes.
It follows that at $q$,  \[0= Q_{\alpha\beta} = \frac{1}{2}\tau_{\alpha\beta} + \mu_{\alpha\beta}.\]  Since $\tau$ and $\mu$ lie in different eigenspaces of the Casimir operator $\Casimir$ by  Proposition \ref{prop:torsionprop}, they must both be zero.  Since $L_{\alpha\beta}$ and $R_{\alpha\beta}$ are determined by $\tau_{\alpha\beta}$, $\mu_{\alpha\beta}$ and $S$, they both vanish at $q$.

Now consider the terms of $Q$ with multiindex of order $3$.  Looking at $Q_{\alpha i}=0$, we immediately see that $T_{\alpha jk}=0$ since $\tensor{\veps}{\up {jk}\down i}$ is an isomorphism from $V$ to $\bigwedge^2V$.

Next we trace $Q_{(\alpha\beta,\gamma)}$ on any two indices to find
\[ 0= \tensor{Q}{\down \alpha \up \alpha \down {,\beta}} + 2\tensor{Q}{\down{\alpha\beta,}\up\alpha} = \frac{(4n+1)(2n+1)}{16n(n+2)}S_{,\beta} + \tensor{\tau}{\down {\alpha\beta,} \up\alpha} + 2\tensor{\mu}{\down{\alpha\beta,} \up \alpha}.\]
Since $T_{\alpha jk}=0$, combining this equation with equations \eqref{eq:taumuS} and \eqref{eq:tormu} shows that at $q$ the tensors $\tensor{\tau}{\down{\alpha\beta,}\up \alpha}$, $\tensor{\mu}{\down{\alpha\beta,}\up \alpha}$ and $S_{,\beta}$ satisfy the system of equations
\[\left(\begin{array}{ccc}1 & 2 & \frac{(4n+1)(2n+1)}{16n(n+2)} \\1 & -6 & -\frac{3}{16n(n+2)} \\1 & 0 & -\frac{3}{16(n+2)}\end{array}\right)\left(\begin{array}{c}\tensor{\tau}{\down{\alpha\beta,}\up \alpha} \\\tensor{\mu}{\down{\alpha\beta,}\up \alpha} \\S_{,\beta}\end{array}\right) = \left(\begin{array}{c}0 \\0 \\0\end{array}\right).\]
The coefficient matrix here is nonsingular and hence $\tensor{\tau}{\down{\alpha\beta,}\up \alpha}$, $\tensor{\mu}{\down{\alpha\beta,}\up \alpha}$ and $S_{,\beta}$ all vanish at $q$.

Moving on to terms of $Q$ with indices of order $4$, we see first that $Q_{ij}= -(1/16n)B_{(ij)}=0$.  Next, tracing $Q_{\alpha i,\beta}$ on the horizontal indices gives 
\[\tensor{Q}{\down{\alpha i,}\up \alpha} = -\tensor{(A^{-1})}{\down{i\alpha}\up{j\beta}} \tensor{T}{\down{\beta kl,}\up\alpha} \tensor{\veps}{\up {kl}\down j}.\]
From equation \eqref{eq:firstbianchi2} and the fact that $\tau_{\alpha\beta}$, $\mu_{\alpha\beta}$ and $T_{\alpha i \beta}$ vanish at $q$, we know that 
\begin{equation} \label{eq:curvandtorsion}
R_{kl\alpha\beta}\tensor{I}{\down j \up {\alpha\beta}} =T_{\beta kl,\alpha} \tensor{I}{\down j \up {\alpha\beta}}.
\end{equation}
Since the automorphism $A$ is parallel at $q$ and has inverse $A^{-1}=(1/8)(A+2)$ and $\tensor{T}{\down{\beta kl,}\up\alpha}$ is antisymmetric in the horizontal indices, we have
\begin{align*}
 \tensor{Q}{\down {\alpha i,}\up\alpha} &= -\tensor{(A^{-1})}{\down {i\alpha}\up {j\beta}} \tensor{T}{\down{\beta kl,}\up\alpha} \tensor{\veps}{\up {kl}\down j} \\
   &= -\frac{1}{4}\tensor{\veps}{\up {jp}\down i}\tensor{I}{\down {p\alpha} \up \beta}\tensor{T}{\down{\beta kl,}\up\alpha} \tensor{\veps}{\up {kl}\down j}\\
   &= -\frac{1}{4}\tensor{\veps}{\up {jp}\down i}\tensor{I}{\down {p\alpha} \up \beta}\tensor{R}{\down{kl\alpha}\up\beta} \tensor{\veps}{\up {kl}\down j}\\
   &= -\frac{1}{4}\tensor{\veps}{\up{jk}\down i} B_{jk}.
\end{align*}
Equation (4.6) from \cite{Vassilevetal:2006} tells us that
\[ \tensor{\veps}{\up {jk}\down i}B_{jk} = -\frac{1}{4n(n+2)} S_{,i},\]
thus tracing $Q_{(\alpha\beta,i)}$ on the horizontal indices yields
\[ 0=\tensor{Q}{\down \alpha \up \alpha \down {,i}} + 2 \tensor{Q}{\down {\alpha i,} \up \alpha} = \frac{4n^2+n+1}{8n(n+2)}S_{,i} .\]
Since the $V$-volume form is an isomorphism,  this also shows that the antisymmetric part of $B_{ij}$ vanishes.  We already know the symmetric part vanishes, so $B_{ij}=0$ at $q$.

Finally, using \eqref{eq:curvandtorsion} and the the fact that the almost complex structures and the $V$-volume form are parallel at $q$, equations \eqref{eq:taumuS}, \eqref{eq:tormu} and \eqref{eq:VHricci} become
\begin{align*}
0&=\tensor{\tau}{\down {\alpha\beta,}\up {\beta\alpha}}  - 6 \tensor{\mu}{\down {\alpha\beta,}\up {\beta\alpha}} + \frac{4n-1}{2}\tensor{B}{\down i \up i} - \frac{3}{16n(n+2)}\tensor{S}{\down{,\alpha}\up\alpha} \\
0&=\tensor{\tau}{\down {\alpha\beta,}\up {\beta\alpha}}  +  \frac{n+2}{2}\tensor{B}{\down i \up i} - \frac{3}{16(n+2)}\tensor{S}{\down{,\alpha}\up\alpha}\\
0  &=\tensor{\tau}{\down {\alpha\beta,}\up {\beta\alpha}}  - 3\tensor{\mu}{\down {\alpha\beta,}\up {\beta\alpha}}  -2\tensor{B}{\down i \up i} - \tensor{R}{\down{\gamma i \beta}\up\gamma\down , \up \alpha}\tensor{I}{\up{i\beta}\down \alpha}.
\end{align*}
We know $Q_{(\alpha\beta,\gamma\delta)}=0$ and so tracing on any two pairs of indices yields
\[0= 2\tensor{Q}{\down{\alpha\beta,}\up{\alpha\beta}} + \tensor{Q}{\down \alpha \up \alpha \down {,\beta}\up \beta} = \tensor{\tau}{\down{\alpha\beta,}\up{\alpha\beta}}+2\tensor{\mu}{\down{\alpha\beta,}\up{\alpha\beta}} + \frac{(4n+1)(2n+1)}{16n(n+2)}\tensor{S}{\down {,\alpha}\up\alpha}.\]
Since $B_{ij}=0$, we therefore have the following system of equations
\[\left(\begin{array}{cccc}1 & -6 & -\frac{3}{16n(n+2)} & 0 \\1 & 0 & -\frac{3}{16(n+2)} & 0 \\1 & -3 & 0 & -1 \\1 & 2 & \frac{(4n+1)(2n+1)}{16n(n+2)} & 0\end{array}\right)\left(\begin{array}{c}\tensor{\tau}{\down{\alpha\beta,}\up{\alpha\beta}} \\\tensor{\mu}{\down{\alpha\beta,}\up{\alpha\beta}} \\\tensor{S}{\down{,\alpha}\up\alpha} \\\tensor{R}{\down{\gamma i \beta}\up\gamma\down , \up \alpha}\tensor{I}{\up{i\beta}\down \alpha}\end{array}\right)=\left(\begin{array}{c}0 \\0 \\0 \\0\end{array}\right).\]
As before, the coefficient matrix is nonsingular, and hence each of $\tensor{\tau}{\down{\alpha\beta,}\up{\alpha\beta}}$, $\tensor{\mu}{\down{\alpha\beta,}\up{\alpha\beta}}$, $\tensor{S}{\down{,\alpha}\up\alpha}$, and $\tensor{R}{\down{\gamma i \beta}\up\gamma\down , \up \alpha}\tensor{I}{\up{i\beta}\down \alpha}$ vanishes at $q$.  This completes the proof.

 \end{proof}
\section{Scalar polynomial invariants} \label{sec:scal_invariants}

\subsection{More normalizations}

A critical next step in the solution to the Yamabe problem is to consider an asymptotic expansion of the Yamabe functional for a suitable class of test functions.  In doing so, we expect to encounter as coefficients certain polynomial tensors in the QC curvature and torsion, to which we are able to assign a weight, defined below.  By analogy with the conformal and CR cases, we expect to be required to consider terms that have weight no more than four.  Further, the work of the preceding section will allow us to show that, at the origin in QC pseudohermitian normal coordinates, our normalizations imply that the only such tensors of weight at most four are dimensional constants and the square norm of the QC conformal curvature tensor \eqref{eq:QCconfcurv}.

Let us now define the weight of a tensor as follows. 
\begin{defn} Suppose $F$ is a homogeneous polynomial in $(x,t)$ whose coefficients are polynomial expressions in the curvature, torsion and the covariant derivatives at $q$.  We define the \emph{weight} $w(F)$ recursively by
\begin{enumerate}
\item $w(T_{abc,D}(q)) = o(bcD) - o(a)$,
\item $w(R_{abcd,E}(q)) = o(abcE)-o(d) = o(abE)$ since $c$ and $d$ always have the same order,
\item $w(F_1F_2) = w(F_1)+w(F_2)$,
\item $w(g_{ab}(q)) = w(g^{ab}(q)) = w(I_{i\alpha\beta}(q)) = w(\veps_{ijk}(q)) = w(c) = 0$,
\item if $w(F_A)=m$ for all $A$, then $w(\sum_A F_A x^A) = m$.
\end{enumerate}
Here $c$ denotes an arbitrary constant, independent of the pseudohermitian structure.  We also let $w(0)=m$ for all $m$.
\end{defn}

We will be interested in the polynomials of weight less than or equal to $4$.  To simplify matters, we recall that $R_{abij}$ is determined by $R_{ab\alpha\beta}$ and the only terms of weight $1$ are identically $0$, namely $T_{\alpha\beta\gamma}$ and $T_{ij\alpha}$.   Table \ref{tab:weights} lists the remaining curvature and torsion terms, organized by weight.
\begin{table}[b]
\centering
\caption{Curvature and torsion terms of weight less than or equal to $4$.}
\begin{tabular}{ccccc}
$0$ & $2$ & $3$ & $4$ \\\hline
$g_{\alpha\beta}$ & $T_{\alpha i \beta}$ & $T_{\alpha ij}$ & $T_{\alpha ij,\beta}$ \\
$g_{ij}$ & $T_{ijk}$ & $T_{\alpha i \beta,\gamma}$ & $T_{\alpha i \beta, \gamma\delta}$ \\
$I_{i\alpha\beta}$ & $R_{\alpha\beta\gamma\delta}$ & $R_{\alpha\beta\gamma\delta,\rho}$ & $T_{\alpha i \beta, j}$ \\
$\veps_{ijk}$ &  & $R_{\alpha i \beta \gamma}$ & $R_{\alpha\beta\gamma\delta,\rho\sigma}$ \\
 &  &  & $R_{\alpha\beta\gamma\delta,i}$ \\
 &  &  & $R_{\alpha i \beta\gamma,\delta}$ \\
 &  &  & $R_{ij\beta\gamma}$
\end{tabular}
\label{tab:weights}
\end{table}
Also, recall that the torsion term $T_{\alpha i \beta}$ is determined by the tensors $\tau_{\alpha\beta}$ and $\mu_{\alpha\beta}$ as in equation \eqref{eq:qctorsion}.

At this point it is convenient to introduce a few more tensors and relations between them from those found in \cite{Vassilevetal:2006}.  These are all found by contracting the curvature tensor against the almost complex structures in various ways.  We define
\begin{equation} \label{eq:ricciforms}
\rho_{iab} = \frac{1}{4n}R_{ab \alpha\beta} \tensor{I}{\down i \up {\beta\alpha}}, \ 
\zeta_{iab} = \frac{1}{4n} R_{\alpha ab \beta}\tensor{I}{\down i \up {\beta\alpha}}, \  
\sigma_{iab} = \frac{1}{4n} R_{\alpha\beta ab} \tensor{I}{\down i \up {\beta\alpha}}.
\end{equation}
These tensors satisfy the following relations, found in \cite[Lemma 3.11]{Vassilevetal:2006} and \cite[Theorem 2.4]{Vassilevetal:2007}.
\begin{prop} \label{prop:ricciprop} We may write the tensors $\rho$, $\zeta$, and $\sigma$ in terms of the torsion and scalar curvature as
\begin{align}
\rho_{i\alpha\beta} &= \frac{1}{2}(\tau_{\alpha\gamma}\tensor{I}{\down i \up \gamma \down \beta} - \tau_{\gamma\beta}\tensor{I}{\down i \up \gamma \down \alpha}) + 2 \mu_{\alpha\gamma}\tensor{I}{\down i \up \gamma\down \beta} - \frac{S}{8n(n+2)}I_{i\alpha\beta}, \label{eq:rho}\\
\zeta_{i\alpha\beta} &= -\frac{2n+1}{4n}\tau_{\alpha\gamma}\tensor{I}{\down i \up \gamma \down \beta} + \frac{1}{4n}\tau_{\gamma\beta}\tensor{I}{\down i \up \gamma \down \alpha} + \frac{2n+1}{2n}\mu_{\alpha\gamma}\tensor{I}{\down i \up \gamma\down \beta} + \frac{S}{16n(n+2)}I_{i\alpha\beta}, \label{eq:zeta}\\
\sigma_{i\alpha\beta} &= \frac{n+2}{2n}(\tau_{\alpha\gamma}\tensor{I}{\down i \up \gamma \down \beta} - \tau_{\gamma\beta}\tensor{I}{\down i \up \gamma \down \alpha}) -  \frac{S}{8n(n+2)}I_{i\alpha\beta}. \label{eq:sigma}
\end{align}
Further, $\rho$ and $\sigma$ are antisymmetric in the horizontal indices, and
\begin{equation}\label{eq:rzsscalar}
\rho_{i\alpha\beta}I^{i\alpha\beta} =  \sigma_{i\alpha\beta}I^{i\alpha\beta}=-\frac{3S}{2(n+2)},\quad \zeta_{i\alpha\beta}I^{i\alpha\beta} = \frac{3S}{4(n+2)}.
\end{equation}
Finally, the curvature tensor $R_{ab\alpha\beta}$ satisfies
\[ R_{ab\alpha\beta} = \mathfrak{R}_{ab\alpha\beta} + \rho_{iab}\tensor{I}{\up i \down {\alpha\beta}},\]
where $\mathfrak{R}_{ab\alpha\beta}$ is the $\mathfrak{sp}(n)$ component of $R_{ab\alpha\beta}$, and hence commutes with the almost complex structures in the second pair of indices.
\end{prop}

As mentioned above, our interest lies in an asymptotic expansion of the Yamabe functional for which we will need to consider scalar pseudohermitian invariants of weight at most $4$.  In particular we would like to know that, for a pseudohermitian structure normalized as in section \ref{sec:coords}, the only interesting terms are constants independent of the structure, and the square norm of the QC conformal curvature tensor.  This is the content of the following
\begin{thm} \label{thm:noscalarterms} Let $M$ be a QC manifold with pseudohermitian structure $\eta$ normalized according to Theorems \ref{thm:symder_vanish} and \ref{thm:mainthm}.  Then, at the center of the normalization, $q$, the only invariant scalar quantities of weight no more than $4$ constructed as polynomials from the invariants listed in Table \ref{tab:weights} are constants independent of the structure and $\norm{W}^2$, the squared norm of the QC conformal curvature tensor; in particular, all other invariant scalar terms vanish at $q$.
\end{thm}
\begin{proof}
We consider the terms by weight.  First we notice that the composition law of the almost complex structures puts an upper bound on the number of such factors that appear in any complete contraction.  Namely, the number of almost complex structures is no more than half the number of horizontal indices, since otherwise, some of the almost complex structures would contract together, resulting in a reduction to fewer such structures by equation \eqref{eq:ACS}.

Before we begin the consideration of the invariants, it is first useful to recall several of the important identities that we have described in previous sections.  In particular we will be using the following equations repeatedly:
\begin{gather*}
\tensor{T}{\up \alpha\down {i\beta}} = \frac{1}{4}(\tensor{\tau}{\up \alpha \down \gamma} \tensor{I}{\down i \up \gamma \down \beta}  + \tensor{I}{\down i \up \alpha \down \gamma} \tensor{\tau}{\up \gamma \down \beta}) + \tensor{I}{\down i \up \alpha \down \gamma}\tensor{\mu}{\up \gamma \down \beta},\tag{\ref{eq:qctorsion}}\\
R_{\alpha\beta}= (2n+2)\tau_{\alpha\beta} + 2(2n+5)\mu_{\alpha\beta} + \frac{S}{4n}g_{\alpha \beta}, \tag{\ref{eq:ricci}}\\
\intertext{and the fact that}
\tensor{A}{\up \alpha \down \beta}\tensor{I}{\down i \up \beta \down \gamma} = \tensor{I}{\down i \up \alpha \down \beta} \tensor{A}{\up \beta \down \gamma} \text{ for any } A\in \mathfrak{sp}(n).
\end{gather*}
We will also need  the identities of Theorem \ref{thm:mainthm} and Proposition \ref{prop:ricciprop}.

For the terms of weight $0$ and $1$, the proposition is clear.  For terms of weight $2$, we know that $T_{\alpha i \beta}$ and $T_{ijk}$ already vanish at $q$ by Proposition \ref{prop:torsionprop}, since the normalizations of Theorem \ref{thm:mainthm} guarantee that $\tau_{\alpha\beta}$, $\mu_{\alpha\beta}$ and $S$ all vanish there.  Then clearly any contractions of these factors also vanish.  From the curvature $4$-tensor, the only complete contraction involving only the metric is the scalar curvature which vanishes at $q$.  The remaining contractions involve two almost complex structures contracted on their vertical indices and yield contractions of the tensors $\rho_{i\alpha\beta}$, $\sigma_{i\alpha\beta}$ and $\zeta_{i\alpha\beta}$ with the almost complex structures.  Then Proposition \ref{prop:ricciprop} shows that these all reduce to multiples of the scalar curvature and hence vanish at $q$.

The terms of weight $3$ are easier to deal with since these factors all have an odd number of horizontal indices.  Since these must be contracted in pairs by either $g_{\alpha\beta}$ or $I_{i\alpha\beta}$, there can be no scalars constructed from them.

Finally, the terms of weight $4$ are the most complicated.  There are two ways in which we can arrive at a term of weight $4$: by taking a product of two factors of weight $2$, or taking a single factor of weight $4$, and then applying terms of weight $0$ to contract to a scalar.  

Let us first consider the case of a product of two factors of weight $2$.  In such a case the only possible contractions involve a square of the curvature, since all the torsion terms vanish even  before taking contractions.  We will break this case into smaller sets based on the number of almost complex structures appearing in the contractions.  Also, we note that because the torsion terms vanish, along with $\rho_{i\alpha\beta}$, the curvature tensor satisfies all the standard Riemannian algebraic Bianchi identities, and commutes with the almost complex structures in either the first or second pair of indices.  Further, the two horizontal indices in any metric or almost complex structure factor must be split between the two curvature factors; if not, one of the curvature factors becomes either a Ricci tensor or one of $\rho_{i\alpha\beta}$, $\sigma_{i\alpha\beta}$, or $\zeta_{i\alpha\beta}$, all of which vanish at $q$.

\begin{enumerate}
\item \emph{No almost complex structures:}  In this case all contractions are handled by the metric.  By the reasoning above, we may assume the first curvature factor is $R_{\alpha\beta\gamma\delta}$ and that the second is the raised index version, with the indices some permutation of $\alpha\beta\gamma\delta$.  Then the antisymmetry in the first and second pair of indices and the ability to switch the first pair with the second pair reduces the 24 possibilities to the following three:
\begin{gather*}
R_{\alpha\beta\gamma\delta}R^{\alpha\beta\gamma\delta} , \\
R_{\alpha\beta\gamma\delta}R^{\alpha\gamma\beta\delta}, \text{ and }
R_{\alpha\beta\gamma\delta}R^{\alpha\delta\beta\gamma}.
\end{gather*}
The first term is the squared norm of the QC conformal curvature tensor, since by \eqref{eq:QCconfcurv} $R_{\alpha\beta\gamma\delta}=W_{\alpha\beta\gamma\delta}$ when $L_{\alpha\beta}=0$.  The last two terms here are negatives of each other since $R_{\alpha\beta\gamma\delta}$ is antisymmetric in $\gamma$ and $\delta$.  Then the algebraic Bianchi identity and vanishing of the torsion at $q$ show that
\begin{equation}\label{eq:usingalgBianchi}
 0= R_{\alpha\beta\gamma\delta}(R^{\alpha\beta\gamma\delta}+R^{\beta\gamma\alpha\delta}+R^{\gamma\alpha\beta\delta}) = \norm{W}^2 -2 R_{\alpha\beta\gamma\delta}R^{\alpha\gamma\beta\delta}.
\end{equation}
Thus the remaining terms are also multiples of the squared norm of the QC conformal curvature tensor.

\item \emph{One almost complex structure:}  Such terms are not possible since they would leave an uncontracted vertical index.

\item \emph{Two almost complex structures:} Two almost complex structures are necessarily contracted on their vertical indices, yielding a Casimir operator.  Further, to contract the remaining horizontal indices we require two metric factors.  Here we take the first factor to be $R_{\alpha\beta\gamma\delta}$ or $R_{\alpha\gamma\beta\delta}$, so that we are always contracting on the $\alpha$ and $\beta$ indices, and leave the second to be some contraction of curvature and the Casimir operator.  By rearranging indices using symmetries, up to a sign, we may assume that the first index of the second curvature factor is contracted against the first index of the first curvature factor.  For the second metric contraction, it may either be on the second indices of both curvature factors, on the third indices of both factors, or up to a sign, between the second index of the first factor and the third index of the second factor.  Considering these possibilities yields the following list:
\begin{subequations}
\begin{gather}
R_{\alpha\beta\gamma\delta}\tensor{R}{\up {\alpha\beta} \down {\mu \nu}} \tensor{I}{\down i \up {\gamma\mu}}I^{i\delta\nu} \label{eq:ACS2.1},\\
R_{\alpha\beta\gamma\delta}\tensor{R}{\up \alpha \down \mu \up \beta \down \nu} \tensor{I}{\down i \up {\gamma\mu}}I^{i\delta\nu} = -
R_{\alpha\beta\gamma\delta}\tensor{R}{\up \alpha \down \nu \up \beta \down \mu} \tensor{I}{\down i \up {\gamma\mu}}I^{i\delta\nu} \label{eq:ACS2.2},\\
R_{\alpha\gamma\beta\delta}\tensor{R}{\up \alpha \down \mu \up \beta \down \nu} \tensor{I}{\down i \up {\gamma\mu}}I^{i\delta\nu} \label{eq:ACS2.3},\\
R_{\alpha\gamma\beta\delta}\tensor{R}{\up \alpha \down \nu \up \beta \down \mu} \tensor{I}{\down i \up {\gamma\mu}}I^{i\delta\nu} \label{eq:ACS2.4}.
\end{gather}
\end{subequations}
The first term is a multiple of $\norm{W}^2$ since $R_{\alpha\beta\gamma\delta}$ is in $\mathfrak{sp}(n)\otimes \mathfrak{sp}(n)$ at $q$ and hence commutes with the almost complex structures.  That is,
\[ R_{\alpha\beta\gamma\delta}\tensor{R}{\up {\alpha\beta} \down {\mu \nu}} \tensor{I}{\down i \up {\gamma\mu}}I^{i\delta\nu} = R_{\alpha\beta\gamma\delta}\tensor{R}{\up {\alpha\beta\gamma} \down {\mu}} \tensor{I}{\down i \up {\mu}\down \nu}I^{i\delta\nu} = 3R_{\alpha\beta\gamma\delta}R^{\alpha\beta\gamma\delta}.\]
The terms \eqref{eq:ACS2.2} are negatives of each other and by an application of the algebraic Bianchi identity as in \eqref{eq:usingalgBianchi} are seen to be proportional to \eqref{eq:ACS2.1}.  In \eqref{eq:ACS2.3} we compute
\begin{equation} \label{eq:usingsymantisym}
\begin{aligned}
R_{\alpha\gamma\beta\delta}\tensor{R}{\up \alpha \down \mu \up \beta \down \nu} \tensor{I}{\down i \up {\gamma\mu}}I^{i\delta\nu}
&=R_{\alpha\gamma\delta\nu}\tensor{R}{\up \alpha \down \mu \up \beta \down \nu} \tensor{I}{\down i \up {\gamma\mu}}I^{i\beta\delta} &&\text{by commuting}\\
&=R_{\alpha\gamma\beta\delta}\tensor{R}{\up \alpha \down \mu \up \nu \down \delta} \tensor{I}{\down i \up {\gamma\mu}}I^{i\nu\beta} &&\text{renaming indices}\\
&=-R_{\alpha\gamma\delta\beta}\tensor{R}{\up \alpha \down \mu \up \delta \down \nu} \tensor{I}{\down i \up {\gamma\mu}}I^{i\beta\nu} &&\text{by symmetries}\\
&=-R_{\alpha\gamma\beta\delta}\tensor{R}{\up \alpha \down \mu \up \beta \down \nu} \tensor{I}{\down i \up {\gamma\mu}}I^{i\delta\nu}. &&\text{renaming indices}
\end{aligned}
\end{equation}
Therefore \eqref{eq:ACS2.3} vanishes at $q$.  Finally, again using the algebraic Bianchi identity as in \eqref{eq:usingalgBianchi}, \eqref{eq:ACS2.4} is seen to be proportional to \eqref{eq:ACS2.1}.

\item \emph{Three almost complex structures:}  Three almost complex structures is a viable option, where the three vertical indices are contract together with $\veps^{ijk}$.  The remaining horizontal indices are contracted with a metric factor, which by symmetries we may assume to be the first index in each curvature term.  Then, using the remaining symmetries we have the following terms:
\begin{subequations}
\begin{gather}
R_{\alpha\beta\gamma\delta} \tensor{R}{\up \alpha \down{\mu\nu\rho}} \tensor{I}{\down i \up {\beta\mu}}\tensor{I}{\down j \up {\gamma\nu}}\tensor{I}{\down k \up {\delta\rho}}\veps^{ijk} \label{eq:ACS3.1},\\
R_{\alpha\beta\gamma\delta} \tensor{R}{\up \alpha \down{\rho\mu\nu}} \tensor{I}{\down i \up {\beta\mu}}\tensor{I}{\down j \up {\gamma\nu}}\tensor{I}{\down k \up {\delta\rho}}\veps^{ijk} \label{eq:ACS3.2},\\
R_{\alpha\beta\gamma\delta} \tensor{R}{\up \alpha \down{\nu\rho\mu}} \tensor{I}{\down i \up {\beta\mu}}\tensor{I}{\down j \up {\gamma\nu}}\tensor{I}{\down k \up {\delta\rho}}\veps^{ijk} \label{eq:ACS3.3}.
\end{gather}
\end{subequations}
Using the fact that the curvature commutes with the almost complex structures in either the first or second pair of indices, an analysis similar to that of equation \eqref{eq:usingsymantisym} shows that \eqref{eq:ACS3.1} vanishes.  By commuting the almost complex structures in the remaining terms with the curvature factors and then contracting the almost complex structures together, we reduce to the case of only two almost complex structures, which has already been dealt with.  For example
\begin{align*}
R_{\alpha\beta\gamma\delta} \tensor{R}{\up \alpha \down{\rho\mu\nu}} \tensor{I}{\down i \up {\beta\mu}}\tensor{I}{\down j \up {\gamma\nu}}\tensor{I}{\down k \up {\delta\rho}}\veps^{ijk}  
&= \tensor{R}{\down{\alpha\beta\gamma}\up\nu} \tensor{R}{\up \alpha \down{\rho\mu\nu}} \tensor{I}{\down i \up {\beta\mu}}\tensor{I}{\down j \up {\gamma}\down\delta}\tensor{I}{\down k \up {\delta\rho}}\veps^{ijk}  \\
&=  \tensor{R}{\down{\alpha\beta\gamma}\up\nu} \tensor{R}{\up \alpha \down{\rho\mu\nu}} \tensor{I}{\down i \up {\beta\mu}}(-\delta_{jk}g^{\gamma\rho} + \veps_{jkl}I^{l\gamma\rho}) \veps^{ijk}\\
&= 2\tensor{R}{\down{\alpha\beta\gamma}\up\nu} \tensor{R}{\up \alpha \down{\rho\mu\nu}}\tensor{I}{\down i \up {\beta\mu}}I^{l\gamma\rho} ,
\end{align*}
again by Proposition \ref{prop:ACSandVform}.  By symmetries of the curvature tensor, this is equivalent to a multiple of \eqref{eq:ACS2.4}.

\item \emph{Four almost complex structures:} Finally, for the case of four almost complex structures we have the following possibilities up to symmetries:
\begin{subequations}
\begin{gather}
R_{\alpha\beta\gamma\delta}R_{\mu\nu\rho\sigma}\tensor{I}{\down i \up {\alpha\mu}}\tensor{I}{\up i \up {\beta\nu}}\tensor{I}{\down j \up {\gamma\rho}}\tensor{I}{\up j \up {\delta\sigma}},\\
R_{\alpha\beta\gamma\delta}R_{\mu\rho\nu\sigma}\tensor{I}{\down i \up {\alpha\mu}}\tensor{I}{\up i \up {\beta\nu}}\tensor{I}{\down j \up {\gamma\rho}}\tensor{I}{\up j \up {\delta\sigma}},\\
R_{\alpha\beta\gamma\delta}R_{\mu\sigma\nu\rho}\tensor{I}{\down i \up {\alpha\mu}}\tensor{I}{\up i \up {\beta\nu}}\tensor{I}{\down j \up {\gamma\rho}}\tensor{I}{\up j \up {\delta\sigma}},\\
R_{\alpha\gamma\beta\delta}R_{\mu\rho\nu\sigma}\tensor{I}{\down i \up {\alpha\mu}}\tensor{I}{\up i \up {\beta\nu}}\tensor{I}{\down j \up {\gamma\rho}}\tensor{I}{\up j \up {\delta\sigma}},\\
R_{\alpha\gamma\beta\delta}R_{\mu\sigma\nu\rho}\tensor{I}{\down i \up {\alpha\mu}}\tensor{I}{\up i \up {\beta\nu}}\tensor{I}{\down j \up {\gamma\rho}}\tensor{I}{\up j \up {\delta\sigma}}.
\end{gather}
\end{subequations}
As in the previous cases, commuting almost complex structures with the curvature factors shows the first term to be a multiple of $\norm{W}^2$, and shows that the remaining terms reduce to the case of only two or three almost complex structures, all of which have been dealt with.
\end{enumerate}

Now we turn to the terms of weight $4$ that do not result from a product of terms of weight $2$.  From $T_{\alpha ij,\beta}$, the only possible contraction is $T_{\alpha ij, \beta}\veps^{ijk}\tensor{I}{\down k \up {\alpha\beta}}$.  But from equations \eqref{eq:defineB} and \eqref{eq:firstbianchi} we see that this expression is $\tensor{B}{\down i \up i} = 0$.

Next, we consider $T_{\alpha i \beta, \gamma\delta}$.  Instead of dealing with this directly, we recall equation \eqref{eq:qctorsion} which tells us that this is determined by $\tau_{\alpha\beta}$, $\mu_{\alpha\beta}$, their covariant derivatives, and the almost complex structures and their covariant derivatives.  Since the almost complex structures are parallel at $q$ and $\tau$ and $\mu$ vanish there, we need only consider the terms $\tau_{\alpha\beta, \gamma\delta}$ and $\mu_{\alpha\beta,\gamma\delta}$.  We deal with $\tau$, since the case for $\mu$ is identical.  We know that $\tau_{\alpha\beta}$ is a trace-free tensor, and so $\tensor{\tau}{\down \alpha \up \alpha \down{,\beta} \up \beta}=0$ locally.  Further by Theorem \ref{thm:mainthm}, $\tensor{\tau}{\down {\alpha\beta,}\up {\alpha\beta}}=0$ at $q$.  Also, tracing $\tau$ against any of the almost complex structures is identically $0$, and so, at $q$
\[ \tau_{\alpha\beta,\gamma\delta} I^{i\alpha\beta}\tensor{I}{\down i \up {\gamma\delta}} = (\tau_{\alpha\beta} I^{i\alpha\beta}\tensor{I}{\down i \up {\gamma\delta}})_{,\gamma\delta} +\ldots\]
Here ``$\ldots$'' denotes terms that contain either $\tau_{\alpha\beta}$ undifferentiated, or terms involving a first derivative of the almost complex structures.  In either case, these terms are $0$ at $q$, and so $\tau_{\alpha\beta,\gamma\delta} I^{i\alpha\beta}\tensor{I}{\down i \up {\gamma\delta}}$ vanishes there.  Finally we may form the trace $\tau_{\alpha\beta,\gamma\delta} I^{i\alpha\gamma}\tensor{I}{\down i \up {\beta\delta}}$.  This term is dealt with similarly, by recalling that $\tau$ is in the $(-1)$-eigenspace of the Casimir operator $\Casimir = I_iI^i$.  As mentioned above, the case for $\mu$ is almost identical, since $\mu$ is trace-free, symmetric and in the $3$-eigenspace of $\Casimir$.

Next in the list is $T_{\alpha i\beta,j}$, which we again deal with by considering $\tau_{\alpha\beta,j}$ and $\mu_{\alpha\beta,j}$.  As above, the two tensors are handled almost identically, and so we present only the case for $\tau$.  The only trace we may form is $\tau_{\alpha\beta,j}I^{j\alpha\beta}$.  But as above, since the almost complex structures are parallel at $q$ we have, at $q$,
\[ \tau_{\alpha\beta,j} I^{j\alpha\beta} = (\tau_{\alpha\beta}I^{j\alpha\beta})_{,j} = 0.\]

Now for the curvature terms of weight $4$.  Beginning with $R_{ij\alpha\beta}$, the only complete contraction is $R_{ij\alpha\beta}\tensor{I}{\down k \up {\alpha\beta}} \veps^{ijk} = \tensor{B}{\down i\up i} = 0$.  Next we consider $R_{\alpha\beta\gamma\delta,i}$.  As mentioned above, we need only consider contractions involving no more than two almost complex structures.  Further, since there is already a vertical index, we must have at least one.  If we have one almost complex structure, we must also have a metric term, and the only possible metric contraction in this case is to $R_{\alpha\beta,i}$.  Since the Ricci tensor is symmetric, we see that $R_{\alpha\beta,i}I^{i\alpha\beta}=0$.  If we admit two almost complex structures, we need to include an $\veps^{ijk}$ term to contract the vertical indices.  Since the almost complex structures are parallel at $q$ we have the following possibilities:
\begin{equation}
\rho_{j\alpha\beta,i}\tensor{I}{\down k \up {\alpha\beta}} \veps^{ijk},\quad 
\sigma_{j\alpha\beta,i}\tensor{I}{\down k \up {\alpha\beta}} \veps^{ijk},\quad
\zeta_{j\alpha\beta,i}\tensor{I}{\down k \up {\alpha\beta}} \veps^{ijk}.
\end{equation}
But at $q$, Proposition \ref{prop:ricciprop} shows that modulo terms that vanish at $q$, these three tensors can be expressed in terms of contractions of $\tau_{\alpha\beta,i}$, $\mu_{\alpha\beta,i}$ and $S_{,i}$, all of which have already been dealt with above.

Next, consider $R_{\alpha i \beta \gamma, \delta}$.  From \cite[Theorem 3.1]{Vassilevetal:2007} we know that 
\begin{multline}
R_{\alpha i \beta \gamma} = \mu_{\delta\gamma,\alpha}\tensor{I}{\down i \up \delta \down \beta} + \frac{1}{4}(\tau_{\gamma\delta,\beta}\tensor{I}{\down i \up \delta \down \alpha} + \tau_{\delta\alpha,\beta} \tensor{I}{\down i \up \delta\down \gamma} - \tau_{\delta \alpha,\gamma} \tensor{I}{\down i \up \delta\down \beta} - \tau_{\beta\delta,\gamma}\tensor{I}{\down i \up \delta \down \alpha})\\ - I_{j\beta\alpha}\tensor{T}{\down \gamma \up j \down i} + I_{j\gamma\alpha}\tensor{T}{\down \beta \up j \down i} + I_{j\gamma\beta}\tensor{T}{\down \alpha \up j \down i}.
\end{multline}
Using this and contracting $R_{\alpha i \beta\gamma,\delta}$ against all possible combinations of weight $0$ terms to form scalar quantities yields terms that contain one of the following:
\begin{itemize}
\item a covariant derivative of an almost complex structure;
\item a contraction of $T_{\alpha ij,\beta}$ that reduces to $\tensor{B}{\down i \up i}$ at $q$ (see \eqref{eq:defineB} and \eqref{eq:firstbianchi});
\item the trace of $T_{\alpha i j ,\beta}$ on the vertical indices;
\item the trace of the action of the Casimir operator on $\tau$ or $\mu$; or
\item a second divergence of $\tau$ or $\mu$.
\end{itemize}
Each of these terms has been considered already in the above analysis and shown to be $0$ at $q$.  Thus, there are no nontrivial scalars to be formed from $R_{\alpha i \beta\gamma,\delta}$.

Finally we come to $R_{\alpha\beta\gamma\delta,\rho\sigma}$.  There are six complete contractions that contain a metric contraction, namely
\begin{gather*} 
\tensor{R}{\down{\alpha\beta,} \up{\alpha\beta}},\quad \tensor{S}{\down {,\alpha} \up \alpha},\quad R_{\alpha\beta,\gamma\delta}\tensor{I}{\down i \up {\alpha\beta}}\tensor{I}{\up{i\gamma\delta}},\quad R_{\alpha\beta,\gamma\delta}\tensor{I}{\down i \up {\alpha\gamma}}\tensor{I}{\up{i\beta\delta}},\\ \tensor{R}{\down{\alpha\beta\gamma\delta,\rho}\up \rho} \tensor{I}{\down i \up {\alpha\beta}}\tensor{I}{\up{i\gamma\delta}},\quad \tensor{R}{\down{\alpha\beta\gamma\delta,\rho}\up \rho} \tensor{I}{\down i \up {\alpha\gamma}}\tensor{I}{\up{i\beta\delta}}.
\end{gather*}
The first two here vanish at $q$ since $\tensor{\tau}{\down{\alpha\beta,}\up {\alpha\beta}}=\tensor{\mu}{\down{\alpha\beta,}\up {\alpha\beta}}=\tensor{S}{\down{,\alpha}\up {\alpha}}=0$ and the Ricci tensor is determined by these tensors via \eqref{eq:ricci}.
The third term vanishes because the horizontal Ricci tensor is symmetric, while the fourth vanishes at $q$ since it represents the action of the Casimir operator on the second covariant derivatives of $\tau$ and $\mu$, which has already been considered above.  The last two terms may be calculated by first tracing with the almost complex structures and then differentiating, since the differences vanish at $q$.  We can then easily compute that these also vanish at $q$.  Notice that all terms that contain a second covariant derivative of the almost complex structures also contain factors like $R_{\alpha\beta\gamma\delta}\tensor{I}{\down i\up{\delta\gamma}} = \rho_{i\alpha\beta}=0$.

The remaining terms are to be found from $R_{\alpha\beta\gamma\delta,\rho\sigma}$ by contracting the indices in pairs with three distinct almost complex structures, and then using $\veps^{ijk}$ to contract to a scalar.  If we contract the two derivative indices with an almost complex structure, we recover only the antisymmetric part of the second covariant derivative, which may be rewritten in terms of only one covariant derivative in the vertical direction.  But notice that we have already taken care of those terms.  By the antisymmetry of the curvature in the first and second pairs of indices, it therefore remains to consider the following four terms
\begin{subequations}\label{eq:ABCD}
\begin{gather} 
A = R_{\alpha\beta\gamma\delta,\rho\sigma} I^{i\alpha\beta}I^{j\sigma\delta}I^{k\rho\gamma}\veps_{ijk},\\
B= R_{\alpha\beta\gamma\delta,\rho\sigma} I^{i\delta\beta}I^{j\sigma\alpha}I^{k\rho\gamma}\veps_{ijk},\\
C=R_{\alpha\beta\gamma\delta,\rho\sigma} I^{i\beta\gamma}I^{j\sigma\delta}I^{k\rho\alpha}\veps_{ijk}, \\
D=R_{\alpha\beta\gamma\delta,\rho\sigma} I^{i\gamma\delta}I^{j\sigma\beta}I^{k\rho\alpha}\veps_{ijk}.
\end{gather}
\end{subequations}

To show that these are all zero we recall the algebraic and differential Bianchi identities for the curvature tensor.  These may be found, for example, in \cite{KN:1996}.  Differentiating the algebraic Bianchi identity twice and the differential Bianchi identity once we have
\begin{multline}
R_{\alpha\beta\gamma\delta,\rho\sigma} + R_{\beta\gamma\alpha\delta,\rho\sigma} + R_{\gamma\alpha\beta\delta,\rho\sigma} =\\ 2 ( T_{\delta i \gamma} \tensor{I}{\up i \down{\beta\alpha}} + T_{\delta i \alpha}\tensor{I}{\up i \down{\gamma\beta}} + T_{\delta i \beta}\tensor{I}{\up i \down{\alpha\gamma}})_{,\rho\sigma}, \label{eq:bianchione}
\end{multline}
\begin{multline}
R_{\alpha\beta\gamma\delta,\rho\sigma} + R_{\beta\rho\gamma\delta,\alpha\sigma} + R_{\rho\alpha\gamma\delta,\beta\sigma} =\\ 2(\tensor{I}{\up i \down{\beta\alpha}}R_{\rho i \gamma\delta} + \tensor{I}{\up i \down {\rho\beta}}R_{\alpha i \gamma \delta} + \tensor{I}{\up i \down{\alpha\gamma}}R_{\beta i \gamma\delta})_{,\sigma} \label{eq:bianchitwo}.
\end{multline}
Contracting the right hand sides of \eqref{eq:bianchione} and \eqref{eq:bianchitwo} with the almost complex structures and $\veps$ as in \eqref{eq:ABCD} yields scalar terms constructed of quantities we have already shown to vanish at $q$.  Contracting the left hand sides then, we arrive at the following equations for $A$, $B$, $C$ and $D$ at $q$,
\[A +2C = 0,\quad B-C-D = 0,\quad A-2C = 0, \quad 2B=0.\]
For example,
\begin{align*}
(R_{\alpha\beta\gamma\delta,\rho\sigma} + R_{\beta\gamma\alpha\delta,\rho\sigma} + &R_{\gamma\alpha\beta\delta,\rho\sigma}) I^{i\alpha\beta}I^{j\sigma\delta}I^{k\rho\gamma}\veps_{ijk} \\
&= A + R_{\beta\gamma\alpha\delta,\rho\sigma} I^{i\alpha\beta}I^{j\sigma\delta}I^{k\rho\gamma}\veps_{ijk} \\
&\qquad+ R_{\gamma\alpha\beta\delta,\rho\sigma} I^{i\alpha\beta}I^{j\sigma\delta}I^{k\rho\gamma}\veps_{ijk} \\
&= A + R_{\alpha\beta\gamma\delta,\rho\sigma}I^{i\gamma\alpha}I^{j\sigma\delta}I^{k\rho\beta}\veps_{ijk}\\
&\qquad + R_{\alpha\beta\gamma\delta,\rho\sigma}I^{i\beta\gamma}I^{j\sigma\delta}I^{k\rho\alpha}\veps_{ijk}\\
&= A - R_{\alpha\beta\gamma\delta,\rho\sigma}I^{i\gamma\beta}I^{j\sigma\delta}I^{k\rho\alpha}\veps_{ijk} + C\\
&= A + R_{\alpha\beta\gamma\delta,\rho\sigma}I^{i\beta\gamma}I^{j\sigma\delta}I^{k\rho\alpha}\veps_{ijk} + C\\ 
&= A + 2C.
\end{align*}
Again, this equals zero since the terms on the left-hand side of equation \eqref{eq:bianchione} have already been show to contract to $0$.  From these equations, it is clear that $A=B=C=D=0$ at $q$.  This concludes the proof.
\end{proof}

%
\bibliographystyle{amsalpha}
\bibliography{biblio}

\end{document}